\documentclass{amsart}
\usepackage{amsmath}
\usepackage{enumerate}
\usepackage{tikz}
\usepackage{pgfplots}
\usepackage{amssymb}
\usepackage{amsfonts}
\usepackage{bm}
\usepackage{marvosym}
\usepackage{caption}
\usepackage{subcaption}
\usepackage{mathrsfs}
\usepackage[OT2,T1]{fontenc}
\usepackage{glossaries}
\usepackage{graphicx}

\newtheorem{thm}{Theorem}[subsection]

\newtheorem{lem}[thm]{Lemma}
\newtheorem{prop}[thm]{Proposition}

\theoremstyle{definition}

\theoremstyle{remark}
\newtheorem{rem}[thm]{Remark}

\numberwithin{equation}{subsection}
\numberwithin{figure}{section}

\newcommand{\C}{{\mathbb C}}
\newcommand{\D}{{\mathbb D}}
\newcommand{\T}{{\mathbb T}}
\newcommand{\R}{{\mathbb R}}

\newcommand{\Hsp}{\mathscr{H}}

\newcommand{\Mop}{\mathbf{M}}

\newcommand{\Lop}{{\mathbf L}}

\newcommand{\Pop}{{\mathbf P}}

\newcommand{\Hop}{{\mathbf H}}

\newcommand{\Uop}{{\mathbf U}}
\newcommand{\Top}{{\mathbf T}}

\newcommand{\im}{\mathrm{Im}}

\newcommand{\calS}{\mathcal{S}}

\newcommand{\calA}{{\mathcal A}}

\newcommand{\calK}{{\mathcal K}}
\newcommand{\calQ}{{\mathscr Q}}

\newcommand{\calE}{{\mathcal E}}

\newcommand{\hDelta}{\varDelta}

\newcommand{\Psiit}{\mathit{\Psi}}

\newcommand{\re}{\mathrm{Re}}
\newcommand{\diff}{{\mathrm d}}
\newcommand{\diffs}{\mathrm{ds}}
\newcommand{\diffA}{\mathrm{dA}}

\newcommand{\imag}{{\mathrm i}}
\newcommand{\Ordo}{\mathrm{O}}

\newcommand{\dbar}{\bar\partial}

\newcommand{\e}{\mathrm e}

\newcommand{\erf}{\operatorname{erf}}

\DeclareSymbolFont{cyrletters}{OT2}{wncyr}{m}{n}
\DeclareFontFamily{U}{rcjhbltx}{}
\DeclareFontShape{U}{rcjhbltx}{m}{n}{<->rcjhbltx}{}
\DeclareSymbolFont{hebrewletters}{U}{rcjhbltx}{m}{n}

\DeclareMathSymbol{\aleph}{\mathord}{hebrewletters}{39}
\DeclareMathSymbol{\beth}{\mathord}{hebrewletters}{98}
\DeclareMathSymbol{\gimel}{\mathord}{hebrewletters}{103}
\DeclareMathSymbol{\daleth}{\mathord}{hebrewletters}{100}
\DeclareMathSymbol{\lamed}{\mathord}{hebrewletters}{108}
\DeclareMathSymbol{\mem}{\mathord}{hebrewletters}{109}
\DeclareMathSymbol{\ayin}{\mathord}{hebrewletters}{96}
\DeclareMathSymbol{\tsadi}{\mathord}{hebrewletters}{118}
\DeclareMathSymbol{\qof}{\mathord}{hebrewletters}{113}
\DeclareMathSymbol{\shin}{\mathord}{hebrewletters}{152}
\DeclareMathSymbol{\memschloss}{\mathord}{hebrewletters}{77}
\DeclareMathSymbol{\nunlange}{\mathord}{hebrewletters}{78}
\DeclareMathSymbol{\vav}{\mathord}{hebrewletters}{79}
\DeclareMathSymbol{\tet}{\mathord}{hebrewletters}{84}
\DeclareMathSymbol{\tsadiklange}{\mathord}{hebrewletters}{90}
\DeclareMathSymbol{\He}{\mathord}{hebrewletters}{104}
\DeclareMathSymbol{\kaf}{\mathord}{hebrewletters}{107}
\DeclareMathSymbol{\nun}{\mathord}{hebrewletters}{110}
\DeclareMathSymbol{\pei}{\mathord}{hebrewletters}{112}
\DeclareMathSymbol{\resh}{\mathord}{hebrewletters}{114}
\DeclareMathSymbol{\samekh}{\mathord}{hebrewletters}{115}
\DeclareMathSymbol{\Het}{\mathord}{hebrewletters}{116}
\DeclareMathSymbol{\vav}{\mathord}{hebrewletters}{119}
\DeclareMathSymbol{\het}{\mathord}{hebrewletters}{120}
\DeclareMathSymbol{\yod}{\mathord}{hebrewletters}{121}
\DeclareMathSymbol{\zayin}{\mathord}{hebrewletters}{122}

\newcommand{\Afun}{\mathrm{A}}
\newcommand{\Bfun}{\mathrm{B}}
\newcommand{\Ffun}{\mathrm{F}}
\newcommand{\Rfun}{\mathrm{R}}
\newcommand{\Hfun}{\mathrm{H}}
\newcommand{\Gfun}{\mathrm{G}}
\newcommand{\Vfun}{\mathrm{V}}
\newcommand\clos{\operatorname{clos}}
\newcommand\apar{a}
\newcommand{\Pol}{\mathrm{Pol}}

\makeglossaries

\begin{document}

\title[Soft Riemann-Hilbert problems
and planar orthogonal polynomials]
{Soft Riemann-Hilbert problems
and planar orthogonal polynomials}

\thanks{This research was supported by Vetenskapsr\aa{}det
(VR grant 2020-03733), by the Leverhulme trust (grant VP1-2020-007),
and by grant 075-15-2021-602 of the Government of the Russian Federation for the 
state support of scientific research, carried out under the supervision 
of leading scientists.}

\author[Hedenmalm]{Haakan Hedenmalm}

\address{Hedenmalm: 1. 
Department of Mathematics
\\
The Royal Institute of Technology
\\
S -- 100 44 Stockholm
\\
SWEDEN.
2. Department of Mathematics and Computer Sciences
\\
St Petersburg State University
\\
St Petersburg 
\\
RUSSIA.
3. Department of Mathematics and Statistics
\\
University of Reading
\\
Reading
\\
United Kingdom
}

\email{haakanh@math.kth.se}






\date{\today}
\begin{abstract}
Riemann-Hilbert problems are jump problems for holomorphic functions
along given interfaces. They arise in various contexts, e.g. in the
asymptotic study of certain nonlinear partial differential equations
and in the asymptotic analysis of orthogonal polynomials. 
Matrix-valued Riemann-Hilbert problems were considered by Deift et al. in
the 1990s with a noncommutative adaptation of the steepest descent method.
For orthogonal polynomials on the line or on the circle with respect to
exponentially varying weights, this led to a strong asymptotic expansion
in the given parameters. 
For orthogonal polynomials with respect to exponentially varying weights in 
the plane, the corresponding asymptotics was obtained by Hedenmalm and 
Wennman (2017), based on the technically involved construction of an
invariant foliation for the orthogonality.  
Planar orthogonal polynomials are characterized in terms of a certain matrix 
$\dbar$-problem (Its, Takhtajan), which we refer to as a \emph{soft 
Riemann-Hilbert problem}. 
Here, we use this perspective to offer a simplified  
approach based not on foliations but instead on the ad hoc insertion of an 
algebraic ansatz for the Cauchy potential in the soft Riemann-Hilbert problem.
This allows the problem to decompose into a hierarchy of scalar
Riemann-Hilbert problems along the interface (the free boundary for
a related obstacle problem). Inspired by microlocal analysis, the method
allows for control of the
solution in such a way that for real-analytic weights, the asymptotics
holds in the $L^2$ sense with error $\Ordo(\e^{-\delta\sqrt{m}})$ in a 
fixed neighborhood of the closed exterior of the interface, for some constant
$\delta>0$,
where $m\to+\infty$. Here, $m$  is the degree of the polynomial, and in terms 
of pointwise asymptotics,  the expansion
dominates the error term in the exterior domain and across the interface 
(by a distance proportional to $m^{-\frac14}$).
In particular, the zeros of the orthogonal polynomial are located 
elsewhere. 
\end{abstract}
\maketitle

\section{The Riemann-Hilbert problem for orthogonal polynomials}

\subsection{Notation}
We write $\C$ for the complex plane, $\D:=\{z\in\C:\,|z|<1\}$ for the open unit disk, 
$\D_\e:=\{z\in\C:\,|z|>1\}$ for the exterior disk,  and $\T=\partial\D$ for the unit circle. 
We let $\diffA$ and $\diffs$ stand for the normalized area and length
measures: $\diffA(z)=\pi^{-1}\diff x\diff y$ and $\diffs(z)=(2\pi)^{-1}|\diff z|$, for 
$z=x+\imag y$. Moreover, we use the complex derivative notation
\[
\partial_z=\tfrac12(\partial_x-\imag\partial_y),\quad 
\dbar_z:=\tfrac12(\partial_x+\imag\partial_y)
\]
where, e.g., $\partial_x$ stands for the partial derivative with 
respect to $x$. Also, we use $\hDelta_z=\partial_z\dbar_z=\frac14(\partial_x^2+\partial_y^2)$
for the Laplacian. When it is not needed for clarity, the subscript $z$ may be omitted.
We talk about the standard classes of smooth functions, $C^k,C^\infty,C^\omega$, for $k$ times
differentiable, infinitely differentiable, and real-analytically smooth functions, respectively. 
Occasionally, we will meet the classes $C^{k,\alpha}$ as well, where $k$ is a nonnegative 
integer and $0<\alpha\le1$. Functions in $C^{k,\alpha}$ have all their partial derivatives up
to and including order $k$ in the H\"older class with exponent $\alpha$ (locally). 

We use $A\ll B$ to denote that $A =\Ordo(B)$ in some limit procedure, and if both $A=\Ordo(B)$
and $B=\Ordo(A)$ then we write $A\asymp B$.

\subsection{The confining potential $Q$}
Let $Q:\C\to\R$ be a $C^2$-smooth function with at least logarithmic growth:
\begin{equation}
Q(z)\ge (1+\varepsilon_0)\log |z|+\Ordo(1)\quad\text{as}\,\,\,|z|\to+\infty,
\label{eq:mingrowthQ}
\end{equation}
for some small positive constant $\varepsilon_0$. We consider the cone $\mathrm{SH}(\C)$ 
of all subharmonic functions $q:\C\to\R\cup\{-\infty\}$, and, given a real parameter $\tau$ with
$0<\tau<+\infty$, the convex subset $\mathrm{Subh}_\tau(\C)$ of functions 
$q\in\mathrm{Subh}(\C)$ with 
\[
q(z)=\tau \log|z|+\Ordo(1) \quad\text{as}\,\,\,|z|\to+\infty.
\]
For $0<\tau<1+\varepsilon_0$, we consider the obstacle problem with $Q$ being the obstacle 
from above, while we optimize over functions in $\mathrm{Subh}_\tau(\C)$ which are dominated 
pointwise by the obstacle. More precisely, we consider the function
\[
\check Q_\tau(z):=\sup\big\{q(z):\,\,q\le Q \,\,\,\text{on}\,\,\,\C, \,\,\,q\in\mathrm{Subh}_\tau(\C)\big\}.
\]
As a matter of definition, $\check Q_\tau\le Q$ holds everywhere. It is a consequence of the general 
methods of obstacle theory that $\check Q_\tau$ is $C^{1,1}$-smooth, which means that its
second order partial derivatives are all locally bounded in the sense of distribution theory. 
Moreover, $\check Q_\tau$ is in $\mathrm{Subh}_\tau(\C)$ as well. The \emph{contact set}
\[
\calS_\tau:=\big\{z\in\C:\,\,\check Q_\tau(z)=Q(z)\big\}
\]
is compact, and off the contact set, $\check Q_\tau$ is harmonic. The contact set $\calS_\tau$ 
will also be referred to as the \emph{spectral droplet}, and it grows with the parameter $\tau$:
\[
\calS_\tau\subset\calS_{\tau'},\qquad 0<\tau<\tau'<1+\varepsilon_0.
\]
Moreover, in the sense of distribution theory, we have the equality 
\[
\hDelta \check Q_\tau=1_{\calS_\tau}\hDelta Q,
\]
where $1_\calE$ denotes the indicator function of $\calE\subset\C$, which equals $1$ 
on $\calE$ and vanishes elsewhere. In particular, $\hDelta Q\ge0$ holds on the spectral droplet 
$\calS_\tau$. 
The growth of $\calS_\tau$ with the parameter $\tau$ is called \emph{weighted Laplacian growth}, 
as the increase in mass in $1_{\calS_\tau}\hDelta Q$ is infinitesimally given by harmonic measure 
for the point at infinity along the interface $\Gamma_\tau:=\partial\calS_\tau$.
For background material, we refer the reader to, e.g., \cite{HM}, \cite{ho}, \cite{HS}, and the book 
\cite{gvt}.

We shall use the potential $Q$ to form the family of
\emph{exponentially varying weights} $\e^{-2mQ}$, where $m$ is a positive real parameter,
and we are interested in asymptotic properties as $m\to+\infty$.  We form the associated
Hilbert space $L^2_{mQ}$ of of equivalence classes of Borel measurable functions with
\begin{equation}
\|f\|_{mQ}:=\langle f,f\rangle_{mQ}^{\frac12}<+\infty,
\label{eq:norm-mQ}
\end{equation}
supplied with corresponding sesquilinear inner product
\begin{equation}
\langle f,g\rangle_{mQ}:=\int_\C f\bar g\,\e^{-2mQ}\diffA.
\label{eq:bilin-mQ}
\end{equation}
For $n=0,1,2,\ldots$, the closed finite-dimensional subspace $\Pol_{mQ,n}$ of $L^2_{mQ}$ 
consisting of polynomials of degree $\le n$ is of interest. If all polynomials of degree 
$\le n$ are represented in $\Pol_{mQ,n}$, that is, if
\begin{equation}
\int_\C (1+|z|^2)^{n}\e^{-2mQ(z)}\diffA(z)<+\infty,
\label{eq:intcond0.01}
\end{equation}
which is a consequence of \eqref{eq:mingrowthQ} for $n<(1+\varepsilon_0)m-1$,
the reproducing kernel of $\Pol_{mQ,n}$ can be used to model a system of $n+1$ particles 
that repel each other by Coulomb interaction, while at the same time being confined by the 
potential $Q$.
To be more precise, the determinantal point process induced by the \emph{correlation kernel} of 
$\Pol_{mQ,n}$ (a variant of the reproducing kernel) models 2D Coulomb gas with a special value
of the temperature parameter in the Gibbs model. For details, see, e.g., \cite{HM}.
The reproducing kernel of $\Pol_{mQ,n}$, call it $K_{m,n}$, may be written in the form
\[
K_{m,n}(z,w)=\sum_{j=0}^{n}e_j(z)\overline{e_j(w)},
\]
where $e_0,\ldots,e_n$ constitutes an orthonormal basis of the space $\Pol_{mQ,n}$.
A particular orthonormal basis is the one that consists of the orthogonal polynomials, normalized
so that each one has norm equal to $1$. Let $P_k$ be the monic orthogonal polynomial of 
degree $k$, which means that the leading coefficient equals $1$:
\[
P_k(z)=z^k+\Ordo(z^{k-1}),\quad\text{as}\,\,\,|z|\to+\infty,
\] 
while $P_k$ is orthogonal to the polynomials of lower degree,
\[
\langle p,P_k\rangle_{mQ}=0,\qquad p\in  \Pol_{mQ,k}.
\]
Then the connection between the orthogonal polynomials and the polynomial reproducing
kernel becomes
\[
K_{m,n}(z,w)=\sum_{j=0}^{n}\|P_j\|_{mQ}^{-2}P_j(z)\overline{P_j(w)}.
\]
In recent work with Wennman \cite{HW1}, an asymptotic expansion formula for the orthogonal 
polynomial  $P_k$ was found, in the instance when $k$ grows in proportion with $m$ as 
$m\to+\infty$. This led to the universal appearance of error function transition for the Coulomb gas 
along the smooth interface $\Gamma_\tau=\partial\calS_\tau$ for $\tau=k/m$ 
(see \cite{HW1}, \cite{HW2}).
For background material, we refer to, e.g., \cite{ahm1}, \cite{ahm2}, and \cite{HM}.  For
related work on the real line in the determinantal case, when orthogonal polynomials are relevant
and enjoy a three-term recursion, 
we refer to \cite{Its1}, \cite{FIK}, \cite{DeiftZhou}, and the book \cite{Deiftbook}.
The proof of the expansion formula in \cite{HW1} is based on the idea of a foliation of a 
neighborhood of the interface $\Gamma_\tau$ by smooth curves where we have (approximately)
the required orthogonality along each curve in the foliation. While that approach has its merits
it may be difficult to expand it beyond the given setting, especially given that the algorithm which 
provides the foliation is rather unwieldy. Here, we supply an alternative approach, which avoids
foliations and instead finds appropriate ad-hoc algebraic formulae for the solution of a
matrix $\dbar$-problem. In so doing, we obtain a simple algebraic framework for the terms of
the asymptotic expansion, expressed in terms of a Neumann series. This permits us to control 
the growth of the terms in the asymptotic expansion using a scale of Banach spaces, analogous
to the Nishida-Nirenberg approach to the Cauchy-Kovalevskaya theorem.

\subsection{A soft Riemann-Hilbert problem for orthogonal polynomials}

We consider a $1\times2$ matrix-valued function $Y=Y(z)$ of the form
\begin{equation}
Y:=\begin{pmatrix}P\,, & \Psiit \\
\end{pmatrix}
\label{eq:Zmatrix}
\end{equation}
where $P$ is an entire functions and $\Psiit$ is $C^1$-smooth
on $\C$. Then
\begin{equation}
\dbar Y=\begin{pmatrix}0\,, & \dbar\Psiit \\
\end{pmatrix}.
\label{eq:Zmatrix2}
\end{equation}
Next, let $W_{mQ}$ be the \emph{soft jump matrix}
\[
W_{mQ}:=\begin{pmatrix}
0 & 
\e^{-2mQ} \\
0 & 0
\end{pmatrix}
\]
and calculate
\[
\bar{Y}W_{mQ}=\begin{pmatrix}\bar P, & \bar{\Psiit} \\
\end{pmatrix}
\begin{pmatrix}
0 & 
\e^{-2mQ} \\
0 & 0
\end{pmatrix}
=\begin{pmatrix}
0\,, & 
\bar P\,\e^{-2mQ} \\
\end{pmatrix}.
\]
Here, we are interested in the asymptotics as $m\to+\infty$. From the 
physical point of view,
the reciprocal $m^{-1}$ plays the role of Planck's constant. 
The corresponding \emph{soft Riemann-Hilbert problem}
\begin{equation}
\dbar Y=\bar Y W_{mQ}
\label{eq:SRH1.Q}
\end{equation}
then amounts to having
\begin{equation}
\dbar\Psiit=
\bar P\,\e^{-2mQ}.
\label{eq:dbar2}
\end{equation}
We would like to couple the soft Riemann-Hilbert problem
\eqref{eq:SRH1.Q} with the following asymptotics at infinity,
for a given nonnegative integer $n$:
\begin{equation}
Y=
\begin{pmatrix}
z^n+\Ordo(|z|^{n-1})\,,& \Ordo(|z|^{-n-1}) \\
\end{pmatrix}
.
\label{eq:SRH2.Q}
\end{equation}
In connection with the Riemann-Hilbert analysis of orthogonal polynomials 
on the real line,
two rows are traditionally used instead of one, and then the asymptotic requirement
\eqref{eq:SRH2.Q} may be written in a more suggestive form (omitted here). 
The condition \eqref{eq:SRH2.Q} involves asymptotics of both $P$ and 
$\Psiit$:
\begin{equation}
P(z)=z^n+\Ordo(|z|^{n-1}),\qquad \Psiit(z)=\Ordo(|z|^{-n-1}),
\label{eq:Ppsicond1}
\end{equation}
as $|z|\to+\infty$. 

\begin{rem}
A word should be said about the terminology. 
After all, Riemann-Hilbert problems are jump problems across 
smooth interfaces, and not $\dbar$-problems like \eqref{eq:SRH1.Q}. 
What speaks in favor of using this terminology is the fact that a holomorphic jump problem 
may be viewed as a $\dbar$-problem in the sense of distribution theory. Indeed, if we 
apply $\dbar$ to a holomorphic function with a jump along an interface, 
the result is, in the sense of distribution theory, a measure supported along the 
interface which amounts to the jump. 
More importantly, in fact, our method of analysis actually decomposes the problem at hand
(given by \eqref{eq:SRH1.Q} and \eqref{eq:SRH2.Q}) into a hierarchy of actual 
Riemann-Hilbert problems, see \cite{HW1}, \cite{HW2}, \cite{HW3}. 
\end{rem}

The appearance of the complex conjugation of the matrix $Y$ on the 
right-hand of \eqref{eq:SRH1.Q} results from the need to characterize the monic
orthogonal polynomial $P$ of degree $n$ with respect to the inner product of 
the Hilbert space $L^2_{mQ}$ given by \eqref{eq:bilin-mQ}.
The complex conjugation does not make the $\dbar$-problem 
any easier, as already the scalar
equation $\dbar u=W\bar u$ is notoriously difficult, see e.g. \cite{KleinMcL}.
The $2\times2$ matrix analogue of the combined problem 
\eqref{eq:SRH1.Q}-\eqref{eq:SRH2.Q} was considered by Its and Takhtajan
\cite{ItsTakhtajan}, and shown to be a way to characterize the orthogonal 
polynomials. We briefly recapture the basic argument. 
Given that $P$ is entire, the first condition in
\eqref{eq:Ppsicond1} says that $P$ is a monic polynomial of degree $n$.
Any solution $\Psi$ to \eqref{eq:dbar2} with the decay
$\Psiit(z)=\Ordo(|z|^{-1})$ at infinity must be of the form of the
\emph{Cauchy potential}
\begin{equation}
\Psiit(z)=
\int_\C\frac{\bar P(\xi)\,\e^{-2mQ(\xi)}}{z-\xi}\diffA(\xi)
\label{eq:Psiint1}
\end{equation}
where $\diffA$ is normalized area measure, provided that the integral
converges. 
The  growth condition \eqref{eq:mingrowthQ} gives that
\begin{equation}
|P(z)|\,\e^{-2mQ(z)}=\Ordo(|z|^{n-2(1+\varepsilon_0)m})\quad\text{as}\,\,\, 
|z|\to+\infty,
\label{eq:massbound1}
\end{equation}
so the condition on $n$ to guarantee the convergence of the integral
 \eqref{eq:Psiint1} is that  
$n<2(1+\varepsilon_0)m-1$. Next, we apply the finite geometric series expansion
\[
\frac{1}{z-\xi}=\frac{1}{z}+\frac{\xi}{z^2}+\ldots+\frac{\xi^{n}}{z^{n+1}}
+\frac{\xi^{n+1}}{z^{n+1}(z-\xi)},
\]
to obtain that
\begin{multline}
\Psiit(z)=\int_\C
\frac{\bar P(\xi)\,\e^{-2mQ(\xi)}}{z-\xi}\diffA(\xi)\\
=\sum_{j=0}^{n}z^{-j-1}\int_\C \xi^j\bar P(\xi)\,\e^{-2mQ(\xi)}\diffA(\xi)
+z^{-n-1}\int_\C\frac{\xi^{n+1}\bar P(\xi)\,\e^{-2mQ(\xi)}}{z-\xi}\diffA(\xi).
\label{eq:Psiint2}
\end{multline}
For this formula to be correct we need all the involved integrals to converge.
In view of \eqref{eq:Psiint1}, this is the case provided that 
$n<(1+\varepsilon_0)m-1$. We also need to estimate the integral
\begin{equation*}
\int_\C\frac{\xi^{n+1}\bar P(\xi)\,\e^{-2mQ(\xi)}}{z-\xi}\diffA(\xi)
\end{equation*}
appearing on the right-hand side of \eqref{eq:Psiint2}. If 
$n\le(1+\varepsilon_0)m-2$ we can estimate 
\[
|\xi^{n+1} P(\xi)|\,\e^{-2mQ(\xi)}\le C (1+|\xi|^2)^{-\frac32}
\]
for some positive constant $C$, so that
\begin{equation}
\bigg|\int_\C\frac{\xi^{n+1}\bar P(\xi)\,\e^{-2mQ(\xi)}}{z-\xi}\diffA(\xi)\bigg|\le C
\int_\C \frac{(1+|\xi|^2)^{-\frac32}}{|z-\xi|}\diffA(\xi)\le 4C.
\label{eq:Psiint3}
\end{equation}
So, we assume that $n\le(1+\varepsilon_0)m-2$ holds.
Identifying terms in \eqref{eq:Psiint2}, it now follows that the second condition in
\eqref{eq:Ppsicond1} amounts to having
\begin{equation}
\int_\C \xi^j\bar P(\xi)\,\e^{-2mQ(\xi)}=0,\qquad j=0,\ldots,n-1,
\label{eq:Psiint4}
\end{equation}
so that $P$ is orthogonal to all polynomials of degree $\le n-1$.
It is now uniquely determined by the first condition in \eqref{eq:Ppsicond1}, i.e.,
that it is a monic polynomial of degree $n$.
In conclusion, the solution $Y$ to \eqref{eq:Zmatrix} with asymptotics
\eqref{eq:SRH2.Q} exists and is uniquely determined: it is given in terms of $P$ and
$\Psiit$, where $\Psiit$ is given by \eqref{eq:Psiint1}, and 
$P=P_n$  is the monic orthogonal polynomial of degree
$n$ with respect to the inner product of $L^2_{mQ}$.


\subsection{Further assumptions on the confining potential $Q$}
\label{ss:further}
If we assume $0<\tau<1+\varepsilon_0$, and that $Q$ is real-analytically smooth 
in a neighborhood of the droplet boundary
$\Gamma_\tau=\partial\calS_\tau$, while $\hDelta Q>0$ holds on $\Gamma_\tau$,
it follows from work of Sakai that $\Gamma_\tau$ consists of real-analytically
smooth arcs with certain types of possible singularities and isolated points. 

To simplify the presentation, we shall focus on the parameter value $\tau=1$.
At $\tau=1$, we now make the following assumptions:

\smallskip

(i) $\Gamma_1$ is a real-analytically smooth Jordan curve, and hence
$\calS_1$ is the simply connected region enclosed by $\Gamma_1$, 

(ii) $Q$ is real-analytically smooth with $\hDelta Q>0$ in a
neighborhood of $\Gamma_1$.

\smallskip

Then $\check Q_1$ equals $Q$ on $\calS_1$ and is harmonic in the
complement $\C\setminus\calS_\tau$, with a $C^{1,1}$-smooth transition
across the interface $\Gamma_1$. Moreover, if we let $\breve Q_1$ denote
the restriction $\check Q_1|_{\C\setminus\calS_1}$, this harmonic function
extends harmonically across $\Gamma_1$, so that $\breve Q_1$ defines a
harmonic function in $\C\setminus\calK$, for some compact subset $\calK$
of the interior of the spectral droplet $\calS_1$. We then form the function
$R=R_1:=Q-\breve Q_1$, which is well-defined and real-valued in
$\C\setminus\calK$. The function $R$ vanishes along with its normal
derivative along the interface $\Gamma_1$ while $\hDelta R=\hDelta Q>0$ holds 
on $\Gamma_1$. The real-analytic smoothness gives that $R$ is real-analytically
smooth near $\Gamma_1$ and we know that $R>0$ on $\C\setminus \calS_1$.
Putting this information together, we see that
$R(z)\asymp [\mathrm{dist}_\C(z,\Gamma_1)]^2$ near $\Gamma_1$,
and that there exists a real-analytically smooth function $\hat R$ in
a neighborhood of $\Gamma_1$ which is a square root of $R$ (so that
$\hat R^2=R$), with $\hat R>0$ on the exterior side of $\Gamma_1$ while
$\hat R<0$ holds on the interior side. This function has a natural
$C^2$-smooth extension to $\C\setminus\calK$, also denoted by $\hat R$,
which keeps the property $\hat R^2=R$, provided the compact subset
$\calK\subset\mathrm{int}\,\calS_1$ is big enough. The extended function
is positive exteriorly to $\Gamma_1$ and negatively interiorly, while
it vanishes precisely along $\Gamma_1$. 

\subsection{Approach to the solution of the soft Riemann-Hilbert problem}
The approach in \cite{HW1} is based on two things: (i) the collapsed orthogonality
equations, and (ii) the  the notion of the \emph{orthogonal foliation flow}.
The collapsed orthogonality equations supply algorithmically the successive terms of the 
expansion formula for the orthogonal polynomial $P$, but are unable to show that the
expansion formula is valid in the first place. The orthogonal foliation flow resolves that
difficulty by supplying an invariant foliation of a neighborhood of $\Gamma_1$ by curves 
on which approximate orthogonality to lower order polynomials holds, and integration
over the flow parameter gives approximate orthogonality with respect to 
$\e^{-2mQ}\diffA$ near $\Gamma_1$. Appropriate glueing procedures, supported in 
part by H\"ormander's $\dbar$-estimates, then show that the approximate orthogonal
polynomial is close to the actual orthogonal polynomial $P$.  The argument is quite 
involved and moreover does not connect $P$ with the $\dbar$-potential $\Psiit$ in 
\eqref{eq:Zmatrix}. Here, we propose a \emph{novel approach} which calculates 
directly candidates for $P$ and $\Psiit$ from a hierarchy of Riemann-Hilbert problems along
$\Gamma_1$, based on an ad-hoc hypothesis on the shape of $P$ and $\Psiit$. Once
the guess is accurate enough, the algorithm kicks in and produces the successive
terms of the asymptotic expansion.
The algorithm also allows for control of the growth of the terms in the expansions as a 
Neumann series, which gives asymptotic control of the error terms, and shows that the 
expansions are valid in a mesoscopic band of width proportional to $m^{-\frac14}$
around $\Gamma_1$, which is substantially wider than the microscopic basic length 
scale $m^{-\frac12}$. The width established in \cite{HW1} was only 
$A m^{-\frac12}\log^{\frac12} m$ for an arbitrarily large positive constant $A$.

\subsection{The ad-hoc shape of $P,\Psiit$ }

We let $\calQ=\calQ_1$ denote the bounded holomorphic function in
$\C\setminus \calS_1$ with $\re \calQ_1=Q$ on $\Gamma_1$ and
$\im \calQ_1(\infty)=0$. Given the real-analytic smoothness of
$Q$ near $\Gamma_1$ and the smoothness of $\Gamma_1$, $\calQ_1$ extends
analytically across $\Gamma_1$. If
\[
\D_\e:=\{z\in\C:\,|z|>1\}
\]
is the punctured exterior disk, and 
$\phi=\phi_1:\C\setminus \calS_1\to\D_\e$ is the conformal mapping which
is onto and preserves the point at infinity, with $\phi'_1(\infty)>0$, then 
\begin{equation}
\breve Q_1=\re\calQ_1+\log|\phi_1|.
\label{eq:Qbreverel}
\end{equation}
In the sequel, we frequently suppress the subscript $1$ and prefer
to write $\phi$ in place of $\phi_1$.  
We solve the soft Riemann-Hilbert problem \eqref{eq:SRH1.Q}
approximately in an iterative fashion. 
The algorithm begins with assuming that $P$ is of the form
\begin{equation*}
P\sim c_{m,n}\phi^n\e^{m\calQ}F, \qquad F=F_0+m^{-1}F_1+m^{-2}F_2+\ldots,
\end{equation*}
where $c_{m,n}$ is a constant, which is fixed by the requirement that 
$F(\infty)=\phi'(\infty)$. In view of the requirement \eqref{eq:Ppsicond1},
 the value of the constant is then
\begin{equation}
c_{m,n}:=(\phi'(\infty))^{-n-1}\e^{-m\calQ(\infty)}>0.
\label{eq:cmn}
\end{equation}
We keep the ratio $\tau=n/m$ fixed, and since we focus on the parameter value 
$\tau=1$, we take $n=m$, so that $m$ is a positive integer which tends to infinity:
\begin{equation}
P\sim c_m\,\phi^m\e^{m\calQ}F, \qquad F=F_0+m^{-1}F_1+m^{-2}F_2+\ldots,
\label{eq:Pform}
\end{equation}
where we use the shorthand notation $c_m:=c_{m,m}$.
For $j=0,1,2,\ldots$, the functions $F_j$ are all supposed to be
independent of the parameter $m$ and as functions bounded and holomorphic
in a domain $\C\setminus\calK$, on which $\calQ$
and $\phi$ are well-defined (holomorphic and conformal, respectively),
for some compact subset $\calK\subset\text{int}\,\calS_1$. In the
representation \eqref{eq:Pform}, we do not really care what happens in
the compact $\calK$, which may seem odd at first glance. However, this is
actually a typical feature of the asymptotics of unweighted Bergman
polynomials studied by Carleman \cite{Carleman} (see also \cite{Carl2}),
following the lead of Szeg\H{o} \cite{szego} (see also \cite{Szeg-book}). 
The analogous instance of a 
fixed weight was only recently analyzed in terms of asymptotic expansion by
Hedenmalm and Wennman \cite{HW3}. 
We should mention that in \eqref{eq:Pform} the series for $F$ need not
converge. Instead, 
the precise statement would be to say that that for any integer
$k=1,2,3,\ldots$,
\begin{equation}
F=F_0+m^{-1}F_1+\ldots+m^{-k}F_k+\Ordo(m^{-k-1}),
\label{eq:Fexpansion}
\end{equation}
with a uniform error bound in $\C\setminus\calK$. In fact, we shall
obtain rather strong control of the growth of
the error term which in turn tells us how big is the optimal $k$ in
the expansion \eqref{eq:Fexpansion}, which gives an expansion with $k$
proportional to $m$ and rapidly decaying error
$\Ordo(\e^{-\epsilon \sqrt{m}})$ for some small positive $\epsilon$. 
The complete analysis requires an analysis of the function $\Psiit$ as well.
We look for $\Psiit$ of the form
\begin{equation}
\Psiit=c_m\,m^{-\frac12}
\phi^{-m}\e^{-m\calQ}
\big\{A \erf(2m^{\frac12}V)+(2\pi m)^{-\frac12}B\chi_1\,\e^{-2mR} \big\},
\label{eq:Psiform}
\end{equation}
where $V$ is a regularized version of $\hat R$: we require that
$V=\hat R$ near $\Gamma_1$ but allow for $V$ to be different further away. 
Here, we use the notation
\[
\erf(x):=\frac{1}{\sqrt{2\pi}}\int_{-\infty}^{x}\e^{-t^2/2}\diff t
\]
for the standard Gaussian error function. Like the function $F$, the
functions $A$ and $B$ are understood in terms of asymptotic expansion,
\begin{equation}
  A=A_0+\ldots+m^{-k}A_k+\Ordo(m^{-k-1}),\quad
  B=B_0+\ldots+m^{-k}B_k+\Ordo(m^{-k-1}).
\label{eq:ABexpansions}
\end{equation}
Like the $F_j$, the functions $A_j$ are all defined and bounded holomorphic in
$\C\setminus\calK$, with asymptotics $\Ordo(|z|^{-1})$ at infinity,
but the $B_j$ are only smooth and also have a more fleeting existence:
they need only be well-defined
in a fixed neighborhood of the loop $\Gamma_1$. The cut-off function $\chi_1$
of \eqref{eq:Psiform} 
solves this issue. It has $0\le\chi_1\le1$ globally, while $\chi_1=1$ on a
fixed small neighborhood of $\Gamma_1$ and $\chi_1=0$ off a slightly bigger
neighborhood of $\Gamma_1$. We make sure the latter neighborhood is contained
strictly inside the region where $B$ is smoothly defined.
This way we get a globally well-defined smooth function $B\chi_1$ by
declaring that it vanishes wherever $B$ is undefined. 

The final step is to show  that the thus defined approximate solution
is close to the true solution $Y=(P\,,\,\,\Psiit)$
where $P$ is the monic orthogonal
polynomial of degree $n$ and $\Psiit$ is given by \eqref{eq:Psiint1}. 

As a side remark, if we think of the expressions
\[
A\erf(2m^{\frac12}V),\quad\text{}\quad
(2\pi m)^{-\frac12}B\,\e^{-2mR}
\]
as quantum waves, then the first is the principal
wave and the second serves as a ``local glue''.  Both types of quantum waves
are needed for the algorithm discussed below.

\subsection{Glimpses of the algorithm} 
\label{ss:glimpses}
Assuming that $\Psiit$ is given by the formula \eqref{eq:Psiform}, where $A$ is
holomorphic, its $\dbar$-derivative equals
\begin{equation}
\dbar \Psiit=(2/\pi)^{\frac12}c_m\,\phi^{-m}\e^{-m\calQ}\big(A\dbar \hat R-
B^1\dbar R+\tfrac12 m^{-1}\dbar B^1\big)\e^{-2mR}
\label{eq:dbar3.0}
\end{equation}
in a neighborhood of $\Gamma_1$, where we agree to the
notation $B^1:=B\chi_1$.
On the other hand, if $P$ is given by
\eqref{eq:Pform}, we may use the fact \eqref{eq:Qbreverel}: 
\begin{equation}
\bar P\,\e^{-2mQ}
=
\bar F{\phi^{-m}\e^{-m\calQ}}\,\e^{-2mR}.
\label{eq:dbar3.01}
\end{equation}
The equation \eqref{eq:dbar2}  requires the expressions \eqref{eq:dbar3.0}
and \eqref{eq:dbar3.01} to coincide, which reduces to the equation
\begin{equation}
A\dbar\hat R-2
B\hat R\,\dbar \hat R+\tfrac12 m^{-1}\dbar B=(\pi/2)^{\frac12}\bar F
\label{eq:dbar3.02}
\end{equation}
in a small neighborhood of $\Gamma_1$, since $B^1=B\chi_1=B$ there.
Note that we used the fact that
$\dbar R=\dbar(\hat R)^2=2\hat R\dbar\hat R$.
We restrict to $\Gamma_1$ and use that $\hat R$ vanishes on the curve:
\begin{equation}
A\dbar\hat R
+\tfrac12 m^{-1}\dbar B=
(\pi/2)^{\frac12}\bar F\quad\text{on}\quad\Gamma_1.
\label{eq:dbar3.03}
\end{equation}
Taken together with the asymptotics \eqref{eq:Ppsicond1}, which
amounts to having the asymptotics as $|z|\to+\infty$
\begin{equation}
F(z)=\phi'(\infty)+\Ordo(|z|^{-1}),\quad
A(z)=\Ordo(|z|^{-1}),
\label{eq:asympFA}
\end{equation}
the equation \eqref{eq:dbar3.03} amounts to a Riemann-Hilbert problem
in an alternative coordinate chart (where the unit circle $\T$
replaces $\Gamma_1$). We explain the
details of this step in Section \ref{sec:algo} below.
For the moment we observe that
the equation involves an unknown $B$, but luckily it is of higher order,
so it does influence the first terms $A_0$ and $F_0$ which get determined
right away by the Riemann-Hilbert problem. 
The original equation \eqref{eq:dbar3.02} contains more information than
\eqref{eq:dbar3.03} alone. Suppose for the moment we were able to solve
the equation \eqref{eq:dbar3.03}. Then we may solve for $B$ using
\eqref{eq:dbar3.02}:
\begin{equation}
B=\frac{A\dbar\hat R+\tfrac12 m^{-1}\dbar B-(\pi/2)^{\frac12}\bar F}
{2\hat R\,\dbar \hat R}.
\label{eq:dbar3.04}
\end{equation}
Since $\hat R$ vanishes only to degree $1$ with $\dbar \hat R\ne0$ along
$\Gamma_1$, the division produces a smooth function, and if the numerator
is real-analytic across $\Gamma_1$, so is the ratio and hence $B$. 
The combination of the Riemann-Hilbert ``jump'' problem
\eqref{eq:dbar3.03} and the smooth division problem \eqref{eq:dbar3.04}
supplies the full algorithm. We may liken it to the Newton algorithm for
finding the zeros of polynomials: once we are in the ballpark the
algorithm gets us ever closer to the solution.  We use
\eqref{eq:dbar3.03} with $B=0$ to get the initial $A$ and $F$. Next,
we apply \eqref{eq:dbar3.04} with the previous choices of $A,F$, and
$B$ to get an updated choice for $B$. This new $B$ is then implemented
into \eqref{eq:dbar3.03} to get improved $A$ and $F$. Proceeding
iteratively we obtain the full asymptotic expansion. 

Our algorithm only gives the asymptotic expansion for the functions
$A,B,F$, and not the functions themselves, with Gevrey class growth of
the coefficients. To pass from asymptotic expansion to actual functions
with the indicated asymptotic behavior, we may e.g. employ the methods
of \cite{HedAM91}. This gives us smooth functions which are holomorphic
whenever the asymptotic expansions are. 

\subsection{Main results}

Apart from the approach based on the ad-hoc shape of the solution as outlined in the
prior two subsections, which is interesting per se and may carry over in other situations, 
it is valuable that we may achieve better control of the error in our approximations than
what was available earlier.
If we write 
\[
F^\approx=F^{\langle\kappa\rangle}:=F_0+m^{-1}F_1+\ldots+m^{-\kappa}F_\kappa,
\]
where $\kappa=\kappa_*(m)\asymp m^{\frac12}$ is appropriately chosen, and the
coefficient functions $F_j$ come from the algorithm presented in Subsection 
\ref{ss:glimpses}, we have the following.

\begin{thm}
We assume that the properties \textrm{(i)-(ii)} of Subsection \ref{ss:further} hold, and put 
$P^\approx:=c_m  \phi^m \e^{m\calQ}F^\approx$, where $F^\approx$ is as above. Then
each $F_j$ extends holomorphically to a fixed neighborhood of  
$\C\setminus\mathrm{int}\,\calS_1$ for each $j=0,\ldots,\kappa_*(m)$, while $\log F_0$ 
is bounded and holomorphic in the same neighborhood.  
Then there exists a fixed $C^\infty$-smooth cut-off function $\chi_{1,1}$ on $\C$, 
with $0\le\chi_{1,1}\le1$, which equals $1$
in an open neighborhood of the closure of $\C\setminus\calS_1$,  and vanishes off
a slightly larger neighborhood, such that $\chi_{1,1}P^\approx $ becomes globally 
well-defined. Moreover, it is close to the monic orthogonal polynomial $P$ of degree $n=m$
in $L^2_{mQ}$:
\[
\|P-\chi_{1,1}P^\approx\|_{mQ}=\Ordo(c_m m^{\frac12}\e^{-\epsilon\sqrt{m}}),\quad
\text{where}\,\,\,\|\chi_{1,1}P^\approx\|_{mQ}\asymp c_m m^{-\frac14}.
\]
Here, $c_m=c_{m,m}$ is given by \eqref{eq:cmn} and $\epsilon>0$ is a constant
that only depends on $Q$. It follows that
\begin{equation*}
P
=c_m\phi^m\e^{m\calQ}\big(F^\approx +\Ordo(m\,\e^{-\frac12\epsilon\sqrt{m}})
\big),\quad\text{on}\,\,\, D_m,
\end{equation*}
holds in the uniform norm as $m\to+\infty$, where $D_m$ is the union of 
$\C\setminus\calS_1$ and a certain band of width $\asymp m^{-\frac14}$ around the loop 
$\Gamma_1=\partial\calS_1$.
\label{thm:main}
\end{thm}

\begin{rem}
The algorithm also calculates the asymptotic expansion of the norm $\|P\|^2_{mQ}$
with an analogously small error term, as we may see from the relations \eqref{eq:orthog4}
and \eqref{eq:fundformula5}, together with the observation that we may replace $\chi_{0,1}$
in \eqref{eq:fundformula5} by its square without any changes in the estimates.
\end{rem}

\subsection{Organization of the paper}
In Section \ref{sec:prel}, we present useful background material: the Bernstein-Walsh lemma,
H\"ormander estimates for the $\dbar$-equation, as well as 
Szeg\H{o} projections and the exterior Herglotz operator.
Then in Section \ref{sec:algo}, we present the algorithm which produces the successive 
approximations, at first intuitively, and later with the inclusion of error terms that come from
an approximate solution of a partial differential equation for a function we denote by $\Bfun$.  
The algorithm dismisses what happens inside a fixed compact part of interior of
the droplet $\calS_1$, so to justify that, we apply the method of $C^\infty$-smooth cut-off 
functions in Section \ref{sec:cutoff}, and solve $\dbar$-equations with growth control to obtain
polynomials. The proof of the main theorem is also supplied in Section \ref{sec:cutoff}, but 
it requires certain estimates that are only developed in the following section, Section 
\ref{sec:AGA}. The method in Section \ref{sec:AGA} involves sequences of nested Banach
spaces of real-analytic functions in neighborhoods of the unit circle $\T$, and resembles the
Nishida-Nirenberg approach to the Cauchy-Kovalevskaya theorem. 

\subsection{Remarks on other related work}
Determinantal processes are used to model the repulsion of fermions. 
This applies to finite as well as infinite particle system, see, e.g., \cite{macchi}. The process 
is induced by a  correlation kernel, often a Bergman kernel, or a polynomial Bergman kernel.
These Bergman kernels enjoy local asymptotic expansions, following work of Tian, Zelditch,
and Catlin (see, e.g., \cite{Zelditch}).  A version which relies on microlocal analysis is supplied 
by Berman, Berndtsson,  and Sj\"ostrand \cite{BBS} (compare with \cite{charles}), while
a stronger expansion theorem with exponentially decaying error term was recently
obtained by Rouby, Sj\"ostrand, and V\~u Ng\d{o}c \cite{RSNg}. It should be mentioned that
while these earlier works are somewhat related to the theme of the present paper, the approaches 
are strictly local and cannot touch the nonlocal phenomena studied here. We also point out that
the literature on local Bergman kernel expansion is extensive and that we only supply a sample 
of the possible references.

A finite particle determinantal process may be thought of in terms of a corresponding Gibbs model 
with a specific thermodynamical inverse temperature $\beta$. As the temperature freezes to $0$, 
the determinantal property is lost, but we expect crystallization to energy minimizing Fekete point 
configurations. Here, we should mention the \emph{Abrikosov conjecture}, which maintains that 
the energetically optimal 
configuration is achieved in the many particle limit by the hexagonal lattice (see, e.g., \cite{serf};
a complete solution is available in the higher dimensional spaces $\R^8$ and $\R^{24}$, 
see \cite{CKMRV}). 
We also mention that a higher-dimensional complex manifold notion of Fekete points was 
considered by Berman, Boucksom, and Witt Nystr\"om \cite{BBWN}.

\subsection{Acknowledgements}
I should thank Aron Wennman for his interest in the approach described here.

\section{Preliminaries}
\label{sec:prel}

We need two basic estimates, the growth estimate which comes from a polynomial 
growth assumption, and an estimate of solutions to $\dbar$-problems with polynomial
growth, which comes from application of H\"ormander's $L^2$-method. We prefer
to formulate both results for a flexible parameter $\tau$ rather than just for $\tau=1$.  
We also introduce the Herglotz operator, which solves a boundary value problem.

\subsection{A weighted Bernstein-Walsh lemma}
We recall Proposition 2.2.2 in \cite{HW1}. The formulation needs the space 
$L^2_{mQ}(\Omega)$, which for an open subset $\Omega\subset\C$ 
is the standard $L^2$-space with finite norm
\[
\|f\|_{mQ,\Omega}^2:=\int_\Omega|f|^2\e^{-2mQ}\diffA<+\infty.
\]

\begin{prop}
Let $\tau=n/m$ be fixed with $0<\tau<1+\varepsilon_0$, and suppose that $\calK$ is a 
compact subset of $\calS_\tau$. Also, fix a positive real parameter $\delta$. 
Then there exists a positive constant $C$ such that 
for any $u\in L^2_{mQ,\C\setminus\calK}$ which is holomorphic on $\C\setminus\calK$
with polynomial growth control $u(z)=\Ordo(|z|^n)$ as $|z|\to+\infty$,  we have that
\[
|u(z)|\le C m^{\frac12}\|u\|_{mQ,\C\setminus\calK}\,\e^{m\check Q_\tau(z)},\qquad
\mathrm{dist}_\C(z,\calK)\ge\delta m^{-\frac12}.
\]
\label{prop:BW}
\end{prop}

\subsection{H\"ormander's estimate for polynomial $\dbar$-problems}
As our method relies on neglecting the behavior of the function we are looking for 
in certain regions of the plane, it is important to be able to reconstruct a globally defined
function. We do this by solving the appropriate $\dbar$-problem, while appealing to the
following basic estimate, which follows from Proposition 2.4.1 in \cite{HW1}.

\begin{prop}
Let $\tau=n/m$ be fixed with $0<\tau<1+\varepsilon_0$, and suppose that $f\in L^\infty(\C)$
vanishes off $\calS_\tau$.
Then there exists a solution $u$ to $\dbar u=f$ with $u(z)=\Ordo(|z|^{n-1})$ as $|z|\to+\infty$
with
 \[
\int_\C|u|^2\e^{-2mQ}\diffA\le  \frac{1}{2m}\int_{\calS_\tau}|f|^2\e^{-2mQ}\frac{\diffA}{\hDelta Q},
\]
provided that the right-hand side is finite.
\label{prop:Horm}
\end{prop}

\subsection{Projections and the Herglotz operator}
\label{ss:Herglotz}
We shall need the Szeg\H{o} projection
\[
\Pop_{H^2}f(z):=\int_\T \frac{f(\zeta)}{1-z\bar\zeta}\diffs(\zeta), \qquad z\in\D.
\]
It is the orthogonal projection $L^2(\T)\to H^2$, where $H^2$ is the standard Hardy space 
in the unit disk, which may be understood both as a space of holomorphic functions in the disk
$\D$ and as a subspace of $L^2(\T)$ via nontangential boundary values. For the
details, see any book on Hardy spaces, e.g., \cite{Duren}, \cite{Garnett}, \cite{Koosis}.
We also need the orthogonal projections to the Hardy subspaces $H^2_0$, $H^2_-$, and 
$H^2_{-,0}$, denoted by $\Pop_{H^2_0}$, $\Pop_{H^2_-}$, and $\Pop_{H^2_{-,0}}$. 
Here, the space $H^2_0$ is the codimension $1$ subspace of $H^2$ consisting of all
functions that vanish at the origin. The space $H^2_-$ may be thought of as a space of 
conjugate-holomorphic functions in the disk, or, alternatively, by
Schwarz reflection, as a space of holomorphic functions in the exterior disk $\D_\e$. Then
$f\in H^2_-$ if and only if the reflected function $\overline{f(1/\bar z)}$ is in $H^2$. Moreover,
the space $H^2_{-,0}$ has codimension $1$ in $H^2_-$, and consists of all functions in 
$H^2_-$ that vanish at the point at infinity. The space $L^2(\T)$ splits orthogonally as 
a direct sum in two ways:
\[
L^2(\T)=H^2_0\oplus H^2_-=H^2\oplus H^2_{-,0}.
\]
The corresponding exterior Szeg\H{o} projections $\Pop_{H^2_-}$ and $\Pop_{H^2_{-,0}}$ may 
be expressed as integrals as well:
\[
\Pop_{H^2_-}f(z)=z\int_\T \frac{f(\zeta)}{z-\zeta}\diffs(\zeta), \qquad 
\Pop_{H^2_{-,0}}f(z)=\int_\T \frac{\zeta f(\zeta)}{z-\zeta}\diffs(\zeta), 
\]
for $z\in\D_\e$. Analogously, we find that
\[
\Pop_{H^2_0}f(z)=z\int_\T \frac{\bar\zeta f(\zeta)}{1-z\bar\zeta}\diffs(\zeta), \qquad z\in\D.
\]
We shall also be interested in the exterior \emph{Herglotz transform}
\[
\Hop_{\D_\e}f(z):=\int_\T \frac{z+\zeta}{z-\zeta}f(\zeta)\diffs (\zeta),\qquad z\in\D_\e.
\] 
Since
\[
\frac{z+\zeta}{z-\zeta}=\frac{2\zeta}{z-\zeta}+1,
\]
it follows that 
\[
\Hop_{\D_\e}f=\Pop_{H^2_-}f+\Pop_{H^2_{-,0}}f=2\Pop_{H^2_{-,0}}+\langle f\rangle_\T,
\]
where we introduce the notation for the mean value on the unit circle,
\begin{equation}
\langle f\rangle_\T:=\int_\T f(\zeta)\diffs(\zeta).
\label{eq:meandef}
\end{equation}
If $f$ is a real-valued function on the circle, say in $L^2(\T)$, then the Herglotz transform 
obtains a solution $u=\Hop_{\D_\e}f$ to the problem $\re\, u=f$ on the circle $\T$, with 
$u$ holomorphic in
$\D_\e$. As such, it is uniquely determined by the requirements that $u\in H^2_-$ and 
$\im\,u(\infty)=0$. 

We may use the Herglotz operator (or the exterior Szeg\H{o} projection) to solve jump
problems. The version we need is Proposition 2.5.1 in \cite{HW1}, formulated below.

\begin{prop}
Suppose that $\vartheta\in L^\infty(\T)$ is of the form $\vartheta=\e^{u+\bar v}$, where
$u,v\in H^\infty$, and let $G$ be a function in $L^2(\T)$. Then $f$ satisfies
\[
f\in H^2_-\cap \vartheta^{-1}(G+H^2)
\]
if and only if
\[
f=C\,\e^{-\bar v}+\e^{-\bar v}\Pop_{H^2_{-,0}}(\e^{-u}G)
\]
for some constant $C$.
\label{prop:Herglotz}
\end{prop}

At times we shall also need the harmonic extension to the exterior disk $\D_\e$ given by
the Poisson integral 
\begin{equation}
\Uop_{\D_\e}f(z):=(|z|^2-1)\int_\T \frac{f(\zeta)}{|z-\zeta|^2}\diffs(\zeta),\qquad z\in\D_\e.
\end{equation}
If $f$ is continuous, then $\Uop_{\D_\e} f$ extends continuously to $\bar\D_\e$, while
if $f$ is real-analytic, then $\Uop_{\D_\e} f$ extends harmonically across $\T$.

\section{Presentation of the the algorithm}
\label{sec:algo}

We first use the conformal mapping $\phi=\phi_1:\C\setminus\calS_1\to\D_\e$
and its inverse $\varphi=\varphi_1=\phi_1^{-1}:\D_\e\to\C\setminus\calS_1$ to turn
our interface loop $\Gamma_1=\partial\calS_1$ into the unit circle $\T$.
Since $\varphi$ remains conformal on a slightly bigger external disk 
\[
\D_\e(0,\rho):=\{z\in\C:\,|z|>\rho\}
\]
with $0<\rho<1$, this means that we discard what happens inside some compact 
subset $\calK$ contained in the interior of the droplet $\calS_1$ (we may think of
$\calK$ as the complement of $\varphi(\D_\e(0,\varrho))$. Later, we justify this
by using cut-off functions to get a globally defined function, and then we apply
H\"ormander's $\dbar$-estimate to get a polynomial solution which is close to the
function we started with.

\subsection{Transfer to the unit circle and the idea of the algorithm}
We write $\varphi=\varphi_1:=\phi_1^{-1}:\D_\e\to\C\setminus\calS_1$
for the indicated conformal mapping, tacitly extended across $\T$, and
consider
\begin{equation*}
\Rfun:=R\circ\varphi,\quad \hat\Rfun:=\hat{R}\circ\varphi,\quad \Psi:=
\Psiit\circ\varphi,
\end{equation*}
and the associated functions
\begin{equation*}
\Afun:=A\circ\varphi,\quad
\Bfun:=B\circ\varphi,\quad \Ffun:=\varphi'F\circ\varphi.  
\end{equation*}
Here, $\Afun$ and $\Ffun$ are holomorphic functions in a neighborhood
of $\bar\D_\e$ with asymptotics in accordance with \eqref{eq:asympFA}:
\begin{equation}
\Ffun(z)=1+\Ordo(|z|^{-1}),\quad
\Afun(z)=\Ordo(|z|^{-1}),
\label{eq:asympFA2}
\end{equation}
as $|z|\to+\infty$. In particular, $\Ffun$ is bounded with a positive
value at infinity, while $A$ is bounded and vanishes at infinity. 
In terms of these functions, the equation \eqref{eq:dbar3.02} reads
\begin{equation}
\Afun\dbar\hat \Rfun-2
\Bfun\hat \Rfun\,\dbar \hat \Rfun+\tfrac12 m^{-1}\dbar \Bfun
=(\pi/2)^{\frac12}\bar \Ffun.
\label{eq:dbar4.01}
\end{equation}
Just as before, we split the equation in two separate steps:
\begin{equation}
\Afun\dbar\hat \Rfun
+\tfrac12 m^{-1}\dbar \Bfun=(\pi/2)^{\frac12}\bar \Ffun
\quad\text{on}\quad\T,
\label{eq:dbar4.02}
\end{equation}
and
\begin{equation}
\Bfun=\frac{\Afun\dbar\hat \Rfun
+\tfrac12 m^{-1}\dbar \Bfun-(\pi/2)^{\frac12}\bar \Ffun}
{2\hat \Rfun\,\dbar \hat \Rfun}.
\label{eq:dbar4.03}
\end{equation}
We will refer to these as {\sc steps I} and {\sc II}, which 
we analyze below in some detail. 
Local analysis around the circle $\T$ shows that
\begin{equation}
\dbar\hat \Rfun=2^{-\frac12}(\varDelta \Rfun)^{\frac12}\zeta
\quad\text{on}
\,\,\,\T.
\label{eq:dbarV}
\end{equation}
Here and in the sequel, $\zeta$ stands for the coordinate function 
$\zeta(z)=z$.
We recall the notation $\Hop_{\D_\e}$ for the exterior Herglotz operator
from Subsection \ref{ss:Herglotz}.
We introduce the function $\Hfun_\Rfun$ given by
\[
\Hfun_\Rfun:=\pi^{-\frac14}\exp\big(\tfrac14
\Hop_{\D_\e}[\log\hDelta \Rfun]\big),
\]
which is a bounded (and bounded away from $0$) holomorphic
function in $\D_\e$ with holomorphic extension across $\T$.
Then
$|\Hfun_\Rfun|^2=\pi^{-\frac12}(\hDelta \Rfun)^{\frac12}$
holds on $\T$, and the value at infinity equals
\[
\Hfun_\Rfun(\infty)=\pi^{-\frac14}
\exp(\tfrac14\langle\log\hDelta \Rfun\rangle_\T)>0.
\] 
It now follows from \eqref{eq:dbarV} that
\begin{equation}
\dbar\hat \Rfun=2^{-\frac12}
\zeta(\varDelta \Rfun)^{\frac12}=
(\pi/2)^{\frac12}\zeta |\Hfun_\Rfun|^2
\quad\text{on}\,\,\,\T.
\label{eq:dbarV2}
\end{equation}
\smallskip

\noindent {\sc Step I}: In view of \eqref{eq:dbarV2}, the equation
\eqref{eq:dbar4.02} may be written as
\begin{equation}
\frac{\Ffun}{\Hfun_\Rfun}=
\overline{\zeta\Afun\Hfun_\Rfun}+
(2\pi)^{-\frac12}m^{-1}\frac{\partial\bar\Bfun}{\Hfun_\Rfun}
\quad\text{on}\,\,\,\T.
\label{eq:dbar10}
\end{equation}
From the data \eqref{eq:asympFA}, we see that in \eqref{eq:dbar10}, 
${\Ffun}/{\Hfun_\Rfun}$ s bounded and holomorphic in
$\bar\D_\e$, so that ${\Ffun}/{\Hfun_\Rfun}\in H^2$,
whereas $\overline{\zeta \Afun \Hfun_\Rfun}$
extends by Schwarz reflection to a bounded 
holomorphic function on $\bar\D$, so that 
$\overline{\zeta \Afun \Hfun_\Rfun}\in H^2_-$. This means that 
\eqref{eq:dbar10} is a Riemann-Hilbert problem with jump 
$(2\pi)^{-\frac12}m^{-1}{\partial\bar\Bfun}/{\Hfun_\Rfun}$, a situation 
covered by Proposition \ref{prop:Herglotz}.
By that result, we may solve for $\Ffun/\Hfun_\Rfun$, 
\begin{equation}
\frac{\Ffun}{\Hfun_\Rfun}=
\apar_1
+(2\pi)^{-\frac12}m^{-1}\Pop_{H^2_{-,0}}\bigg[\frac{\partial\bar \Bfun}{\Hfun_\Rfun}\bigg],
\label{eq:solF0}
\end{equation}
where $\apar_1$ is a constant. 
At the same time,
\begin{equation}
\zeta \Afun\Hfun_\Rfun=
\bar\apar_1
-(2\pi)^{-\frac12}m^{-1}\Pop_{H^2_{-}}\bigg
[\frac{\dbar\Bfun}{\bar \Hfun_\Rfun}\bigg],
\label{eq:solA0}
\end{equation}
while the constant $\apar_1$ is given by 
\begin{equation}
\apar_1=\bigg\langle\frac{\Ffun}{\Hfun_\Rfun}\bigg\rangle_\R=
\frac{\Ffun(\infty)}{\Hfun_\Rfun(\infty)}=\frac{1}{\Hfun_\Rfun(\infty)}
=\pi^{\frac14}
\exp\big(-\tfrac14\langle\log\hDelta \Rfun\rangle_\T\big)>0,
\label{eq:solF0.02}
\end{equation}
since $\Ffun(\infty)=1$ is our assumed normalization and $\Hfun_\Rfun(\infty)$ is
known.
This ends our analysis of {\sc step I}.

\smallskip

\noindent {\sc Step II}: We turn to the analysis of the next step,
based on the the division formula \eqref{eq:dbar4.03}.
To this end, we apply first \eqref{eq:dbarV2}, to see that we may write
\begin{equation}
\dbar\hat \Rfun=(\pi/2)^{\frac12}\zeta|\Hfun_\Rfun|^2
+2W_\Rfun\hat \Rfun \dbar\hat\Rfun,
\label{eq:dbarRdecomp}
\end{equation}
where the expression $W_\Rfun$ is real-analytic near $\T$.
We recall that $\Uop_{\D_\e}$ stands for the harmonic extension to $\D_\e$ of the
restriction to $\T$ of a given smooth function.
It is an immediate consequence of \eqref{eq:dbar10} that
\begin{equation}
\zeta\Afun\Hfun_\Rfun=\Uop_{\D_\e}[\zeta\Afun\Hfun_\Rfun]
=\frac{\bar \Ffun}{\bar \Hfun_\Rfun}-
(2\pi)^{-\frac12}m^{-1}\Uop_{\D_\e}\bigg[
\frac{\dbar\Bfun}{\bar\Hfun_\Rfun}\bigg]
\label{eq:dbar10.1}
\end{equation}
holds in a neighborhood of the closed exterior disk $\bar\D_\e$, in the
sense that each term extends harmonically across $\T$. 
It now follows from \eqref{eq:dbarRdecomp} and \eqref{eq:dbar10.1}
that
\begin{equation*}
\Afun\dbar \hat\Rfun+\tfrac12 m^{-1}\dbar \Bfun
-(\pi/2)^{\frac12}\bar \Ffun
=\tfrac12 m^{-1}\big(\dbar \Bfun-\bar \Hfun_\Rfun
\Uop_{\D_\e}(\dbar \Bfun/\bar \Hfun_\Rfun)\big)
+2\Afun \hat\Rfun\dbar\hat\Rfun\, W_\Rfun,
\end{equation*}
so that \eqref{eq:dbar4.03} asserts that
\begin{equation}
\Bfun=\tfrac12 m^{-1}\frac{\dbar\Bfun-\bar \Hfun_\Rfun
\Uop_{\D_\e}(\dbar \Bfun/\bar \Hfun_\Rfun)}
{2\hat\Rfun\dbar \hat\Rfun}
+\Afun W_\Rfun.
\label{eq:Bsolution0}
\end{equation}

\subsection{The equation for the function $\Bfun$ alone}

In view of the expression for the solution to the
Riemann-Hilbert problem \eqref{eq:solA0} from {\sc Step I}
and the relation \eqref{eq:solF0.02},  
the expression for $\Bfun$ may be written as
\begin{equation*}
\Bfun
=\tfrac12 m^{-1}\frac{\dbar \Bfun-\bar \Hfun_\Rfun
\Uop_{\D_\e}(\dbar \Bfun/\bar \Hfun_\Rfun)}
{2\hat\Rfun\dbar \Rfun}
+\frac{W_\Rfun}{\zeta \Hfun_\Rfun}\bigg(
\apar_1-(2\pi)^{-\frac12}m^{-1}\Pop_{H^2_{-}}\bigg
[\frac{\dbar \Bfun}{\bar \Hfun_\Rfun}\bigg]\bigg),
\end{equation*}
since $\apar_1>0$, which we simplify to
\begin{equation}
\Bfun=\apar_1\frac{W_\Rfun}{\zeta \Hfun_\Rfun}
+m^{-1}\bigg(\frac{\dbar \Bfun-\bar \Hfun_\Rfun
\Uop_{\D_\e}[\dbar \Bfun/\bar \Hfun_\Rfun]}
{4\hat\Rfun\dbar \hat\Rfun}
-(2\pi)^{-\frac12}\frac{W_\Rfun}{\zeta\Hfun_\Rfun}
\Pop_{H^2_{-}}\bigg
[\frac{\dbar \Bfun}{\bar \Hfun_\Rfun}\bigg]\bigg).
\label{eq:Bequation1}
\end{equation}
Note that in \eqref{eq:Bequation1}, we have merged
{\sc steps I} and II in a single equation which only
concerns the function $\Bfun$ (the function $\Afun$
no longer appears). 
This is an operator equation of the type
\begin{equation}
\Bfun=\Bfun_{0}+m^{-1}\Top[\Bfun],\quad \text{where}\quad
\Bfun_0:=\apar_1\frac{W_\Rfun}{\zeta \Hfun_\Rfun}
\quad\text{and}\quad\Top[\Bfun]
:=\Lop[\dbar \Bfun/\bar\Hfun_\Rfun],
\label{eq:NN1}
\end{equation}
and the operator $\Lop$ is given by
\begin{equation}
\Lop[f]:=\bar\Hfun_\Rfun\frac{f-\Uop_{\D_\e}[f]}
{4\hat\Rfun\dbar\hat\Rfun}-(2\pi)^{-\frac12}
\frac{W_\Rfun}{\zeta\Hfun_\Rfun}\Pop_{H^2_-}[f].
\label{eq:Lopdef1.001}
\end{equation}
We may attempt to solve the equation \eqref{eq:NN1} by
iteration, that is, by finite Neumann series approximation. 

\subsection{Quantitative work: approximate solutions $\Afun^\approx,
\Bfun^\approx,\Ffun^\approx$}
In the preceding subsection, we were a little naive, believing that
the equation \eqref{eq:NN1} would have an exact solution.
As for now, we only expect the equation \eqref{eq:NN1} to be 
approximately solvable.
So, we suppose we have found an approximate solution $\Bfun^\approx$,
with
\begin{equation}
\Bfun^\approx=\Bfun_0+m^{-1}\Top[\Bfun^\approx]+\calE_m.
\label{eq:approxB}
\end{equation}
where $\calE_m$ measures the error in solving the equation \eqref{eq:NN1}. 
Later on, in Section \ref{sec:AGA}, we will show that there exists a $\Bfun^\approx$
such $\calE_m$ decays rapidly as $m\to+\infty$.
Naturally, $\Bfun^\approx$ also depends on $m$, but is more important that the
error term reflects this. 
Based on our previously derived equations \eqref{eq:solF0} and \eqref{eq:solA0},  
we put accordingly
\begin{equation}
\Afun^\approx:=\frac{\apar_1}{\zeta\Hfun_\Rfun}-m^{-1}
\frac{(2\pi)^{-\frac12}}{\zeta\Hfun_\Rfun}
\Pop_{H^2_-}\bigg[\frac{\dbar \Bfun^\approx}{\bar\Hfun_\Rfun}\bigg]
\label{eq:Abeta1}                                                                                
\end{equation}
and 
\begin{equation}
\Ffun^\approx:=\apar_1\Hfun_\Rfun
+(2\pi)^{-\frac12}m^{-1}\Hfun_\Rfun\Pop_{H^2_{-,0}}\bigg
[\frac{\partial\bar \Bfun^\approx}{\Hfun_\Rfun}\bigg].
\label{eq:solF0.11}
\end{equation}
In particular, it follows that $\Ffun^\approx$ is properly normalized at infinity,
as
\[
\Ffun^\approx (\infty)=\apar_1\Hfun_\Rfun(\infty)=1,
\]
in view of \eqref{eq:solF0.02}.
It follows from the relations \eqref{eq:Abeta1} and \eqref{eq:solF0.11}
as well as \eqref{eq:solF0.02} that 
\begin{equation}
\frac{\bar\Ffun^\approx}{\bar\Hfun_\Rfun}
-\zeta\Hfun_\Rfun\Afun^\approx=(2\pi)^{-\frac12}m^{-1}\Uop_{\D_\e}
\bigg[\frac{\dbar\Bfun^\approx}{\bar\Hfun_\Rfun} \bigg],
\label{eq:Laurent}
\end{equation}
so that 
\begin{equation}
\zeta|\Hfun_\Rfun|^2 \Afun^\approx=
\bar\Ffun^\approx
-(2\pi)^{-\frac12}m^{-1}\bar\Hfun_\Rfun\Uop_{\D_\e}
\bigg[\frac{\dbar\Bfun^\approx}{\bar\Hfun_\Rfun} \bigg].
\label{eq:Laurent2}
\end{equation}
We see from this and \eqref{eq:dbarRdecomp} that
\begin{multline*}
\Afun^\approx\dbar\hat\Rfun=
(\pi/2)^{\frac12}\zeta|\Hfun_\Rfun|^2 \Afun^\approx
+2 W_\Rfun \hat\Rfun\dbar\hat\Rfun\,\Afun^\approx
\\
=
(\pi/2)^{\frac12}\bar\Ffun^\approx
-\tfrac12 m^{-1}\bar\Hfun_\Rfun\Uop_{\D_\e}
\bigg[\frac{\dbar\Bfun^\approx}{\bar\Hfun_\Rfun} \bigg]
+2 W_\Rfun \hat\Rfun\dbar\hat\Rfun\,\Afun^\approx,
\end{multline*}
so that
\begin{multline*}
\Afun^\approx\dbar\hat\Rfun+\tfrac12 m^{-1}\dbar\Bfun^\approx
\\
=(\pi/2)^{\frac12}\bar\Ffun^\approx
+\tfrac12 m^{-1}\bigg(\dbar\Bfun^\approx-
\bar\Hfun_\Rfun\Uop_{\D_\e}
\bigg[\frac{\dbar\Bfun^\approx}{\bar\Hfun_\Rfun} \bigg]\bigg)
+2 W_\Rfun \hat\Rfun\dbar\hat\Rfun\,\Afun^\approx.
\end{multline*}
Finally, if we apply the relation \eqref{eq:approxB} which defines 
the error $\calE_m$, we find that
\begin{multline}
\Afun^\approx\dbar\hat\Rfun+\tfrac12 m^{-1}\dbar\Bfun^\approx
-\Bfun^\approx\dbar\Rfun
\\
=(\pi/2)^{\frac12}\bar\Ffun^\approx
+\tfrac12 m^{-1}\bigg(\dbar\Bfun^\approx-
\bar\Hfun_\Rfun\Uop_{\D_\e}
\bigg[\frac{\dbar\Bfun^\approx}{\bar\Hfun_\Rfun} \bigg]\bigg)
\\
+2 W_\Rfun \hat\Rfun\dbar\hat\Rfun\,\Afun^\approx-\Bfun_0\dbar\Rfun
-m^{-1}(\dbar\Rfun)\Top[\Bfun^\approx]-\calE_m\dbar\Rfun.
\label{eq:Laurent5}
\end{multline}
Since $\dbar\Rfun=2\hat\Rfun\dbar\hat\Rfun$, we read off the
definition of $\Top$ that
\begin{equation*}
(\dbar\Rfun)\Top[\Bfun^\approx]=\frac12\bigg(\dbar\Bfun^\approx
-\bar\Hfun_\Rfun \Uop_{\D_\e}\bigg[\frac{\dbar\Bfun^\approx}
{\bar\Hfun_\Rfun}\bigg]\bigg)
-(2\pi)^{-\frac12}\frac{W_\Rfun\dbar R}{\zeta\Hfun_\Rfun}
\Pop_{H^2_-}\bigg[\frac{\dbar\Bfun^\approx}{\bar\Hfun_\Rfun}\bigg],
\end{equation*}
so that \eqref{eq:Laurent5} simplifies to
\begin{multline}
\Afun^\approx\dbar\hat\Rfun+\tfrac12 m^{-1}\dbar\Bfun^\approx
-\Bfun^\approx\dbar\Rfun
\\
=(\pi/2)^{\frac12}\bar\Ffun^\approx
+(2\pi)^{-\frac12}m^{-1}\frac{W_\Rfun\dbar R}{\zeta\Hfun_\Rfun}
\Pop_{H^2_-}\bigg[\frac{\dbar\Bfun^\approx}{\bar\Hfun_\Rfun}\bigg]
+W_\Rfun \dbar\Rfun\,\Afun^\approx-\Bfun_0\dbar\Rfun
-\calE_m\dbar\Rfun
\\
=(\pi/2)^{\frac12}\bar\Ffun^\approx
-\calE_m\dbar\Rfun,
\label{eq:Laurent6}
\end{multline}
where in the last step we used the definitions of $\Bfun_0$ and $\Afun^\approx$.


\section{The usefulness of cut-off functions}
\label{sec:cutoff}


The method outlined in the previous subsection only supplies a tentative 
solution to the soft Riemann-Hilbert problem of \eqref{eq:SRH1.Q} and 
\eqref{eq:SRH2.Q}.  So far, however, we have not even been able to assert
much rigorously. For instance, it is not clear that the approximate solution 
$Y=(P\,,\,\,\,\Psiit)$ supplied in terms of the triple 
$(\Afun^\approx,\Bfun^\approx,\Ffun^\approx)$ is even close to the actual solution.
We will explain why the function $\Ffun^\approx$ defines a good 
approximation $P^\approx$ of the orthogonal polynomial $P$.
 
 \subsection{Cut-off functions and a modification of the potential $\hat R$}  
 Let $\chi_0$ be a \emph{radial} cut-off function, which is $C^\infty$-smooth in
the plane with $0\le\chi_0\le1$ globally: $\chi_0=1$ on a fixed (closed) annular 
neighborhood $\calA_1$ of $\T$, while $\chi_0=0$ off a slightly larger fixed open 
annular neighborhood $\calA_0$.  We ask of $\chi_0$ furthermore that
$|\dbar\chi_0|^2\le C_1\chi_0$ holds for some large positive constant $C_1$.
The closure $\clos\calA_0$ is the support of $\chi_0$, and 
$\calA_1\subset\calA_0\subset\clos\calA_0$.
We let $\Vfun\ge0$ be a smooth potential in the plane such that 
$\Vfun=\hat\Rfun$ holds on $\clos\calA_0$, the support of  $\chi_0$, and that $V=0$
holds only along the circle $\T$, and $\Vfun$ is Lipschitz continuous in the plane, with
minimal growth $\Vfun(z)\ge \log|z|$ for $|z|\gg1$.
It is not difficult to see that
\begin{equation}
\int_\C(1-\chi_0)|\dbar\Vfun|\,\e^{-2m\Vfun^2+m\Vfun^2_-}\diffA\le C_2
m^{-1}\e^{-\varepsilon_1 m},
\label{eq:Vest1}
\end{equation}  
for some positive constants $C_2$ and $\varepsilon_1$ and big enough $m$. 
Here, $\Vfun^2_-:=1_\D\Vfun^2$ equals $\Vfun^2$ in the disk $\D$ while
it vanishes outside. The property \eqref{eq:Vest1} might not hold with
$\Vfun$ replaced by $\hat\Rfun$, and it is the reason why we need to introduce 
$\Vfun$ in the first place. 
It is not too difficult to construct a viable $\Vfun$; the necessary details are left to the
interested reader.
We put
\begin{equation}
\Psi^\approx=\Afun^\approx\,\erf(2\sqrt{m}\Vfun)+(2\pi m)^{-\frac12} 
\chi_0 \,\Bfun^\approx\,\e^{-2m\Rfun},
\label{eq:Phiform3}
\end{equation}
and use the function $\Psi^\approx$ to define the corresponding approximate function
\[
\Psiit^\approx:=c_m m^{-\frac12}\phi^{-m} \e^{-m\calQ}\Psi^\approx\circ\phi,
\]
which is then well-defined and smooth in the whole exterior domain to the curve 
$\Gamma_1$ as well as across the interface $\Gamma_1$ (so that it is smooth
in a neighborhood of $\Gamma_1$).
We take a radius $\rho_1$ with $0<\rho_1<1$ and the associated exterior disk
\[
\D_\e(0,\rho_1):=\{z\in\C:\,\rho_1<|z|<+\infty\}.
\] 
We require the support $\clos\calA_0$ of the cut-off function
$\chi_0$ to be contained inside the exterior disk
$\D_\e(0,\rho_1)$, and moreover, that $\rho_1$ is
so close to $1$ that $\hat\Rfun$ is smooth and well-defined in the annulus 
$\{z\in\C:\,\rho_1\le|z|\le1\}$. In addition, we require of $\Vfun$ that $\Vfun=\hat\Rfun$
holds there in addition to the annulus $\clos\calA_0$.

\subsection{Application of the Cauchy-Green formula}
Let $\Gfun$ be a bounded holomorphic function in the exterior disk 
$\D_\e(0,\rho_1)=\{z\in\C:\,|z|>\rho_1\}$, subject to the growth bound
\begin{equation}
|\Gfun(z)|\le \|G\|_{m\Rfun_-} 
\e^{m\Rfun_-(z)},\qquad z\in\D_\e(\rho_1,0),
\label{eq:Gbound}
\end{equation}
where $\|\Gfun\|_{m\Rfun_-}$ denotes the smallest such constant and $\Rfun_-:=1_\D\Rfun$.
We now take the $\dbar$-derivative in \eqref{eq:Phiform3} and apply the identity 
\eqref{eq:Laurent6} as well as the fact that $\Vfun=\hat\Rfun$ on the support of
$\chi_0$: 
\begin{multline}
\dbar \Psi^\approx=
(2m/\pi)^{\frac12}\big[\chi_0\Afun^\approx\dbar \hat\Rfun- 
\chi_0 \Bfun^\approx\dbar \Rfun+\tfrac12 m^{-1}\chi_0\dbar \Bfun^\approx
\\+
\tfrac12 m^{-1}\Bfun^\approx\dbar\chi_0
\big]\,\e^{-2m\Rfun}+(2m/\pi)^{\frac12}(1-\chi_0)\Afun^\approx \dbar\Vfun\,\e^{-2mV^2}
\\
=(2m/\pi)^{\frac12}\big((\pi/2)^{\frac12}\chi_0\bar \Ffun^\approx
+\tfrac12 m^{-1}\Bfun^\approx\dbar\chi_0-\chi_0\calE_m\dbar\Rfun
\big)\,\e^{-2m\Rfun}
\\
+(2m/\pi)^{\frac12}(1-\chi_0)\Afun^\approx \dbar\Vfun\,\e^{-2mV^2}.
\label{eq:dbar3.1}
\end{multline} 
Since $\Vfun(z)$ approaches $+\infty$ as $|z|\to+\infty$, the potential $\Psi^\approx$
has a limit at infinity, which is given by
\begin{equation}
(\zeta\Psi^\approx)(\infty)=(\zeta\Afun^\approx)(\infty)
=\frac{\apar_1}{\Hfun_\Rfun(\infty)}
-m^{-1}\frac{(2\pi)^{-\frac12}}{\Hfun_\Rfun(\infty)}\bigg\langle
\frac{\dbar\Bfun^\approx}{\bar\Hfun_\Rfun}\bigg\rangle_\T.
\label{eq:Phiinfty}
\end{equation}
In view of the Cauchy-Green formula, we have
\begin{equation*}
\Gfun(\infty)(\zeta\Psi^\approx)(\infty)
-\frac{1}{2\pi\imag}\int_{\T(0,\rho_1)} \Gfun\Psi^\approx\diff z
=\int_{\D_\e(0,\rho_1)} \Gfun\dbar\Psi^\approx \diffA,
\end{equation*}
which we rewrite in the form
\begin{multline}
m^{\frac12}\int_{\D_\e(0,\rho_1)}\Gfun\chi_0
\bar\Ffun^\approx\e^{-2mR}\diffA
\\
=\Gfun(\infty)(\zeta\Psi^\approx)(\infty)
-\frac{1}{2\pi\imag}\int_{\T(0,\rho_1)} \Gfun\Psi^\approx\diff z
-\int_{\D_\e(0,\rho_1)}\Gfun\,\big(\dbar\Psi^\approx
-m^{\frac12}\chi_0\bar\Ffun^\approx\e^{-2m\Rfun}\big)
\diffA
\\
=\Gfun(\infty)\Psi^\approx(\infty)
-\frac{1}{2\pi\imag}\int_{\T(0,\rho_1)} \Gfun\Afun^\approx
\erf(2m^{\frac12}\hat\Rfun)\diff z
\\
-(2m/\pi)^{\frac12}
\int_{\D_\e(0,\rho_1)}\Gfun\,\Big\{\tfrac12 m^{-1}\Bfun^\approx
\dbar\chi_0\,\e^{-2m\Rfun}
+(1-\chi_0)\Afun^\approx \dbar\Vfun\, \e^{-2m\Vfun^2}
\Big\}\diffA
\\
+(2m/\pi)^{\frac12}
\int_{\D_\e(0,\rho_1)}\Gfun\,\chi_0\calE_m\dbar\Rfun\,\e^{-2m\Rfun}
\diffA,
\label{eq:fundformula1}
\end{multline}
where in the last step we used that $\chi_0=0$ and $\Vfun=\hat\Rfun$
on the circle $\T(0,\rho_1)$, as well as the formul\ae{}
\eqref{eq:Phiform3} and \eqref{eq:dbar3.1}. We now turn to estimating
the expressions on the right-hand side of \eqref{eq:fundformula1}.
A first observation is that in the integral along the circle
$\T(0,\rho_1)$, $\hat\Rfun<0$, so that $\erf(2m^{\frac12}\hat\Rfun)$ is
small there. To obtain a precise estimate, we may apply integration by
parts to see that for $x>0$,
\[
\erf(-x)=(2\pi)^{-\frac12}x^{-1}\e^{-x^2/2}
-(2\pi)^{-\frac12}\int_x^{+\infty}t^{-2}\e^{-t^2/2}\diff t<
(2\pi)^{-\frac12}x^{-1}\e^{-x^2/2},
\]
so that
\[
\erf(2m^{\frac12}\hat\Rfun)<(8m\pi)^{-\frac12}\Rfun^{-\frac12}
\e^{-2m\Rfun}\quad\text{on}\quad\T(0,\rho_1).
\]
This gives that
\begin{multline}
\bigg|\frac{1}{2\pi\imag}\int_{\T(0,\rho_1)} \Gfun\Afun^\approx
\erf(2m^{\frac12}\hat\Rfun)\diff z\bigg|
\\
\le
(8m\pi)^{-\frac12}\|\Gfun\|_{m\Rfun_-}
\|\Afun^\approx\|_{L^\infty(\T(0,\rho_1))}
\|\Rfun^{-\frac12}\e^{-m\Rfun}\|_{L^\infty(\T(0,\rho_1))}
\\
=\Ordo\big(\|\Gfun\|_{m\Rfun_-}
\|\Afun^\approx\|_{L^\infty(\T(0,\rho_1))}\,m^{-\frac12}\e^{-\varepsilon_2 m}\big),
\label{eq:lineintest1}
\end{multline}
where the positive constant $\varepsilon_2$ may be taken as the minimum
of $\Rfun$ on the circle $\T(0,\rho_1)$, and the implicit ``O''
constant remains bounded as $m\to+\infty$. This offers an exponential
decay estimate of one term on the right-hand side of
\eqref{eq:lineintest1}. To handle the next term, we use the estimate
\begin{multline}
\bigg|\int_{\D_\e(0,\rho_1)} \Gfun\Bfun^\approx\dbar\chi_0\,
\e^{-2m\Rfun}\diffA\bigg|
\\
\le
\|\Gfun\|_{m\Rfun_-}\|\Bfun^\approx\|_{L^\infty(\calA_0)}
\|\dbar\chi_0\|_{L^\infty(\C)}\int_{\calA_0\setminus\calA_1}
\e^{-m\Rfun}\diffA
\\
=\Ordo\big(\|\Gfun\|_{m\Rfun_-}
\|\Bfun^\approx\|_{L^\infty(\calA_0)}\,
\e^{-\varepsilon_3 m}\big),
\label{eq:lineintest2}
\end{multline}
where the positive constant $\varepsilon_3$ may be taken as the infimum of
$\Rfun$ on the set $\calA_0\setminus\calA_1$, and the implicit
``O'' constant remains bounded as $m\to+\infty$. The following term
is handled similarly:
\begin{multline}
\bigg|\int_{\D_\e(0,\rho_1)} (1-\chi_0)\Gfun\Afun^\approx\dbar\Vfun\,
\e^{-2m\Vfun^2}\diffA\bigg|
\\
\le
\|\Gfun\|_{m\Rfun_-}\|\Afun^\approx\|_{L^\infty(\T(0,\rho_1))}
\int_{\D_\e(0,\rho_1)}(1-\chi_0)|\dbar\Vfun|\e^{-2m\Vfun^2+m\Vfun_-^2}\diffA
\\
\le C_2\|\Gfun\|_{m\Rfun_-}
\|\Afun^\approx\|_{L^\infty(\T(0,\rho_1))}\,m^{-1}\e^{-\varepsilon_1 m},
\label{eq:lineintest3}
\end{multline}
in view of \eqref{eq:Vest1} and the maximum principle applied to the
bounded holomorphic function $\Afun^\approx$ on $\D_\e(0,\rho_1)$.
The last term on the right-hand side of \eqref{eq:fundformula1} involves
the error term $\calE_m$, and we estimate it as follows:
\begin{multline}
\bigg|\int_{\D_\e(0,\rho_1)}\Gfun\chi_0\calE_m \dbar\Rfun\,\e^{-2m\Rfun}\diffA
\bigg|
\\
\le \|\Gfun\|_{m\Rfun_-}\|\calE_m\|_{L^\infty(\calA_0)}
\int_{\calA_0}|\dbar\Rfun|\,\e^{-2m\Rfun+m\Rfun_-}\diffA
\\
=\Ordo\big(m^{-\frac12}\|\Gfun\|_{m\Rfun_-}\|\calE_m\|_{L^\infty(\calA_0)}\big),
\label{eq:lineintest4}
\end{multline}
which tells us about the  need to estimate the error $\calE_m$ uniformly
on the annulus $\calA_0$. 
For the bookkeeping, we apply the obtained
estimates \eqref{eq:lineintest1}, \eqref{eq:lineintest2},
\eqref{eq:lineintest3}, and \eqref{eq:lineintest4} in the context of
\eqref{eq:fundformula1}, to obtain that
\begin{multline}
m^{\frac12}\int_{\D_\e(0,\rho_1)}\Gfun\chi_0
\bar\Ffun^\approx\e^{-2mR}\diffA
\\
=\apar_1\frac{\Gfun(\infty)}{\Hfun_\Rfun(\infty)}
-\frac{(2\pi)^{-\frac12}G(\infty)}{m\Hfun_\Rfun(\infty)}
\bigg\langle\frac{\dbar\Bfun^\approx}{\bar \Hfun_\Rfun}\bigg\rangle_\T
+\Ordo\big(\|\Gfun\|_{m\Rfun_-}\|\Afun^\approx\|_{L^\infty(\T(0,\rho_1))}
\,\e^{-\varepsilon_4m}\big)
\\
+\Ordo\big(\|\Gfun\|_{m\Rfun_-}
\|\Bfun^\approx\|_{L^\infty(\calA_0)})\,m^{-\frac12}\e^{-\varepsilon_4 m}\big)+
\Ordo\big(\|\Gfun\|_{m\Rfun_-}\|\calE_m\|_{L^\infty(\calA_0)}\big),
\label{eq:fundformula2}
\end{multline}
where $\varepsilon_4$ stands for the least of the positive constants
$\varepsilon_1,\varepsilon_2,\varepsilon_3$, and we have applied the identity
\eqref{eq:Phiinfty}.
If $\Gfun(\infty)=0$, and if the norm of $\calE_m$ decays quickly to
$0$ as $m\to+\infty$,  it follows from \eqref{eq:fundformula2} that
$\chi_0\Ffun^\approx$ is approximately orthogonal to $\Gfun$ in the
space 
$L^2(\D_\e(0,\rho_1),\e^{-2m\Rfun}\diffA)$.
Finally, we would like to turn this into an approximate orthogonality
between the functions $\Ffun^\approx$ and $G$, but involving a related
smooth cut-off function which avoids making a smooth cut-off in the exterior
disk $\D_\e$. The cut-off function $\chi_0$ factors
$\chi_0=\chi_{0,0}\chi_{0,1}$, where $\chi_{0,j}$ are radial and
$C^\infty$-smooth with $0\le\chi_{0,j}\le1$, and $\chi_{0,0}$
decreases as the radius increases, whereas $\chi_{0,1}$ instead
increases. Then $\chi_{0,0}=\chi_{0,1}=1$ on the annulus $\calA_1$,
so that in particular $\chi_{0,1}=1$ holds on $\D_\e$ while
$\chi_{0,1}=1$ on $\D$. Moreover, for $j=1,2$, 
$|\dbar\chi_{0,j}|^2\le C_1\chi_{0,j}$ holds
with the same constant as for $\chi_0$. The identity
\begin{multline}
\int_{\D_\e(0,\rho_1)}\Gfun\chi_{0,1}\bar\Ffun^\approx \e^{-2m\Rfun}\diffA
\\
=\int_{\D_\e(0,\rho_1)}\Gfun\chi_{0}\bar\Ffun^\approx \e^{-2m\Rfun}\diffA
+\int_{\D_\e(0,\rho_1)}\Gfun\,(1-\chi_{0,0})\chi_{0,1}\bar\Ffun^\approx
\e^{-2m\Rfun}\diffA
\label{eq:testfnid1}
\end{multline}
together with the estimate
\begin{multline}
\int_{\D_\e(0,\rho_1)}\Gfun\,(1-\chi_{0,0})\chi_{0,1}
\bar\Ffun^\approx\e^{-2mR}\diffA
\\
=\Ordo\big(\|\Gfun\|_{m\Rfun_-}
\|\Ffun^\approx\|_{L^\infty(\T(0,\rho_1))}\,
m^{-\frac12}\e^{-\varepsilon_5 m}\big),
\label{eq:fundformula3}
\end{multline}
tells us that we may think instead of $\chi_{0,1}\Ffun^\approx$
as being approximately orthogonal to $\Gfun$ provided that $G(\infty)=0$.
The estimate \eqref{eq:fundformula3} 
holds for some $\varepsilon_5>0$ because $\Rfun$ is bounded below
by a positive constant on $\D_\e(0,\rho_1)\setminus\calA_1$, and
because $\Rfun$ has some minimal logarithmic growth at infinity.

\subsection{Control of the error terms}

By a combination of \eqref{eq:fundformula2}, \eqref{eq:testfnid1},
and \eqref{eq:fundformula3}, we obtain that 
\begin{multline}
m^{\frac12}\int_{\D_\e(0,\rho_1)}\Gfun\,\chi_{0,1}\,
\bar\Ffun^\approx \e^{-2m\Rfun}\diffA
\\
=\frac{\Gfun(\infty)}{\Hfun_\Rfun(\infty)^2}
-\frac{(2\pi)^{-\frac12}G(\infty)}{m\Hfun_\Rfun(\infty)}
\bigg\langle\frac{\dbar\Bfun^\approx}{\bar \Hfun_\Rfun}\bigg\rangle_\T
+\Ordo\big(\|\Gfun\|_{m\Rfun_-}\|\Afun^\approx\|_{L^\infty(\T(0,\rho_1))}
\,\e^{-\varepsilon_6m}\big)
\\
+\Ordo\big(\|\Gfun\|_{m\Rfun_-}
\|\Bfun^\approx\|_{L^\infty(\calA_0)})\,m^{-\frac12}\e^{-\varepsilon_6 m}\big)+
\Ordo\big(\|\Gfun\|_{m\Rfun_-}
\|\Ffun^\approx\|_{L^\infty(\T(0,\rho_1))}\,\e^{-\varepsilon_6 m}\big)
\\
+\Ordo\big(\|\Gfun\|_{m\Rfun_-}\|\calE_m\|_{L^\infty(\calA_0)}\big),
\label{eq:fundformula4}
\end{multline}
if $\varepsilon_6>0$ is the minimum of $\varepsilon_4$ and $\varepsilon_5$. 
Here, we applied the formula for $\apar_1$ of \eqref{eq:solF0.02}.
We may apply the identity with $\Gfun=\Ffun^\approx$, while recalling the
normalization $\Ffun^\approx(\infty)=1$, and since, by the maximum principle,
\[
\|\Gfun\|_{m\Rfun_-}\le 
\|\Gfun\|_{L^\infty(\T(0,\rho_1))},
\]
it follows from \eqref{eq:fundformula4} that
\begin{multline}
m^{\frac12}\int_{\D_\e(0,\rho_1)}\chi_{0,1}
|\Ffun^\approx|^2 \e^{-2m\Rfun}\diffA
\\
=\frac{1}{\Hfun_\Rfun(\infty)^2}-
\frac{(2\pi)^{-\frac12}}{m\Hfun_\Rfun(\infty)}
\bigg\langle\frac{\dbar\Bfun^\approx}{\bar \Hfun_\Rfun}\bigg\rangle_\T
+\Ordo\big(\|\Ffun^\approx\|_{L^\infty(\T(0,\rho_1))}
\|\Afun^\approx\|_{L^\infty(\T(0,\rho_1))}
\e^{-\varepsilon_6m}\big)
\\
+\Ordo\big(\|\Ffun^\approx\|_{L^\infty(\T(0,\rho_1))}
\|\Bfun^\approx\|_{L^\infty(\calA_0)}\,m^{-\frac12}
\e^{-\varepsilon_6 m}\big)+
\Ordo\big(
\|\Ffun^\approx\|^2_{L^\infty(\T(0,\rho_1))}\,
\e^{-\varepsilon_6 m}\big)
\\
+\Ordo\big(\|\Ffun^\approx\|_{L^\infty(\T(0,\rho_1))}
\|\calE_m\|_{L^\infty(\calA_0)}\big),
\label{eq:fundformula5}
\end{multline}
which tells us how to compute the left-hand side norm 
once we may control the error term $\calE_m$. Incidentally, it follows
that up to small error, the mean value 
\[
\bigg\langle\frac{\dbar\Bfun^\approx}{\bar \Hfun_\Rfun}\bigg\rangle_\T 
\]
must be real-valued. If we write $G:=\phi'\Gfun\circ\phi$ and
$F^\approx=\phi'\Ffun^\approx\circ\phi$, which means that we think
that the two functions transform as $(1,0)$-forms, and put
$\chi_{1,1}:=\chi_{0,1}\circ\phi$,  the change-of-variables formula
gives that
\begin{equation}
\int_{\D_\e(0,\rho_1)}\Gfun\,\chi_{0,1}\,
\bar\Ffun^\approx \e^{-2m\Rfun}\diffA=
\int_{\varphi(\D_\e(0,\rho_1))}G\,\chi_{1,1}\,
\bar F^\approx \e^{-2mR}\diffA
\label{eq:orthog1}
\end{equation}
and that
\begin{equation}
\int_{\D_\e(0,\rho_1)}\chi_{0,1}\,
|\Ffun^\approx|^2 \e^{-2m\Rfun}\diffA=
\int_{\varphi(\D_\e(0,\rho_1))}\chi_{1,1}\,
|F^\approx|^2 \e^{-2mR}\diffA.
\label{eq:orthog2}
\end{equation}
We note that the integrals on the right-hand sides of \eqref{eq:orthog1}
and \eqref{eq:orthog2} may be thought to extend to all of $\C$
by declaring that the integrand vanishes in the complement of
$\varphi(\D(0,\rho_1))$. The extended integrands are then
$C^\infty$-smooth on $\C$. 
We put $F^\approx:=\phi'\,\Ffun^\approx\circ\phi$, and write in 
accordance with \eqref{eq:Pform}, 
\[
P^\approx:=c_m\phi^m\e^{m\calQ}F^\approx,
\quad g:=\phi^m\e^{m\calQ}
G,
\]
which we think of as the approximate monic orthogonal polynomial of
degree $m$ and a general approximate polynomial of degree $\le m$,
respectively. 
We find that
\begin{equation}
\int_{\D_\e(0,\rho_1)}\Gfun\,\chi_{0,1}\,
\bar\Ffun^\approx \e^{-2m\Rfun}\diffA=
c_m^{-1}\int_{\C}g\,\chi_{1,1}\,
\bar P^\approx \e^{-2mQ}\diffA,
\label{eq:orthog3}
\end{equation}
while 
\begin{equation}
\int_{\D_\e(0,\rho_1)}\chi^2_{0,1}\,
|\Ffun^\approx|^2 \e^{-2m\Rfun}\diffA=
c_m^{-2}\int_{\C}\chi^2_{1,1}\,
|P^\approx|^2 \e^{-2mQ}\diffA.
\label{eq:orthog4}
\end{equation}
Note that in \eqref{eq:orthog4} we square the cut-off
function, which is all right since the square shares all the 
essential features with the cut-off function itself. Most
importantly, \eqref{eq:fundformula5} holds with $\chi_{0,1}$
replaced by its square $\chi_{0,1}^2$. 
We shall apply the identity \eqref{eq:orthog3} 
to an actual polynomial $g$ of degree $\le m-1$. In this case,
$\Gfun(\infty)=0$, so that the first term on the right-hand side
of \eqref{eq:fundformula4} vanishes and all that remains are the
error terms. To simplify the further discussion we recall the
notation $\Pol_n=\mathrm{Pol}_{mQ,n}$ for the subspace of all polynomials
with degree $\le n$ in the Hilbert space
$L^2_{mQ}=L^2(\C,\e^{-2mQ})$.
We normalize $g$ to have norm $1$ in the space
$\mathrm{Pol}_{mQ,m-1}$.
The Bernstein-Walsh estimate of Proposition \ref{prop:BW}
applies in particular to polynomials of degree $\le m$, and reads 
\begin{equation}
|h(z)|\le C_Q\,\|h\|_{mQ} m^{\frac12}\e^{m\check Q_1(z)},\qquad
z\in\C,
\label{eq:BW1}
\end{equation}
for some positive constant $C_Q$ and every $h\in\mathrm{Pol}_{mQ,m}$. 
As we apply this estimate to  $g\in\mathrm{Pol}_{mQ,m-1}$, it follows that
\[
|\Gfun|\le C_Q m^{\frac12}|\varphi'|\,\e^{m\Rfun_-}.
\]
If $\rho_1<1$ is close enough to $1$, the derivative
$\varphi'$ is uniformly bounded in $\D_\e(0,\rho_1)$, so that in
view of the above, $\|\Gfun\|_{m\Rfun_-}=\Ordo(m^{\frac12})$. 
By classical duality, we have that
\begin{equation*}
\sup_{g\in\mathrm{Pol}_{m-1}:
\,\|g\|_{mQ}=1}\bigg|\int_\C g\,\chi_{1,1}\bar P^\approx
\e^{-2mQ}\diffA\bigg|= \inf_{h\in (\mathrm{Pol}_{m-1})^\perp}
\big\|\chi_{1,1}P^\approx-h\big\|_{mQ}. 
\end{equation*}
and if we combine this with \eqref{eq:fundformula4},  we arrive at
\begin{multline}
\inf_{h\in (\mathrm{Pol}_{m-1})^\perp}
\big\|\chi_{1,1}P^\approx-h\big\|_{mQ}
=\Ordo\big(c_m\|\Afun^\approx\|_{L^\infty(\T(0,\rho_1))}
\,m^{\frac12}\e^{-\varepsilon_6m}\big)
\\
+\Ordo\big(c_m
\|\Bfun^\approx\|_{L^\infty(\calA_0)}\,\e^{-\varepsilon_6 m}\big)+
\Ordo\big(c_m
\|\Ffun^\approx\|_{L^\infty(\T(0,\rho_1))}\,m^{\frac12}\e^{-\varepsilon_6 m}\big)
\\
+\Ordo\big(c_m\,m^{\frac12}\|\calE_m\|_{L^\infty(\calA_0)}\big).
\label{eq:duality002}
\end{multline}

\subsection{Application of H\"ormander's $\dbar$-estimate}
\label{ss:Horm}
The next step follows the lines of \cite{HW1}. We want to turn the
approximate polynomial $P^\approx$ into an actual polynomial.
We put 
\[
P^\ast:=\chi_{1,1}P^\approx-u_1,
\]
where $u_1$ solves the $\dbar$-problem
\[
\dbar u_1=P^\approx\dbar\chi_{1,1}.
\]
As before, $\chi_{1,1}P^\approx$ is extended to vanish wherevever 
$P^\approx$ gets undefined, and hence defines a $C^\infty$-smooth
function in the complex plane $\C$. Clearly, 
$\dbar P^\ast=0$ holds throughout and hence $P^\ast$ defines an 
entire function.
By Proposition \ref{prop:Horm}, based on H\"ormander's classical 
$\dbar$-estimate (which in itself is a dualized version of the corresponding
Carleman estimate), there exists a $u_1$ of
growth $u_1(z)=\Ordo(|z|^{n-1})$ at infinity, such that
\begin{equation}
\|u_1\|^2_{mQ}\le \frac{1}{2m}\int_{\calS_1} |P^\approx|^2|\dbar\chi_{1,1}|^2
\frac{\e^{-2mQ}}{\varDelta Q}\diffA.
\label{eq:Hormest}
\end{equation}
In view of the growth bound on $u_1$, $P^\ast$ must be a polynomial of
degree $n$,  and its leading coefficient equals $1$, since 
$F^\approx(\infty)=1$ and hence, by the definition of $c_m$,
\[
\lim_{|z|\to+\infty} z^{-n}P^\approx(z)=c_m \phi'(\infty)^m
\e^{m\calQ(\infty)}F^\approx(\infty)=1.
\]
After a little bit of rewriting, we find that
\[
\int_{\calS_1} |P^\approx|^2|\dbar\chi_{1,1}|^2
\frac{\e^{-2mQ}}{\varDelta Q}\diffA=c_m^2\int_{\D_\e(0,\rho_1)}|\Ffun^\approx|^2
|\dbar\chi_{0,1}|^2\frac{\e^{-2m\Rfun}}{\varDelta\Rfun}\diffA.
\]
This expression decays exponentially as $m\to+\infty$, as $\Rfun$ is strictly 
positive on the support of $\dbar\chi_{0,1}$. It now follows from 
\eqref{eq:Hormest} that
\begin{equation}
\|u_1\|_{mQ}=\Ordo\big(c_m \|\Ffun^\approx\|_{L^\infty(\T(0,\rho_1)}
\,m^{-1}\e^{-\varepsilon_7 m}\big)
\label{eq:uest1}
\end{equation}
for a positive constant $\varepsilon_7>0$. This procedure gives us a monic polynomial 
$P^\ast$ of degree $n$ with 
\begin{multline}
\inf_{h\in (\mathrm{Pol}_{m-1})^\perp}
\big\|P^\ast-h\big\|_{mQ}
=\Ordo\big(c_m\|\Afun^\approx\|_{L^\infty(\T(0,\rho_1))}
\,m^{\frac12}\e^{-\varepsilon_8m}\big)
\\
+\Ordo\big(c_m
\|\Bfun^\approx\|_{L^\infty(\calA_0)}\,\e^{-\varepsilon_8 m}\big)+
\Ordo\big(c_m
\|\Ffun^\approx\|_{L^\infty(\T(0,\rho_1))}\,m^{\frac12}\e^{-\varepsilon_8 m}\big)
\\
+\Ordo\big(c_m\,m^{\frac12}\|\calE_m\|_{L^\infty(\calA_0)}\big),
\label{eq:duality003}
\end{multline}
if $\varepsilon_8$ is the least of $\varepsilon_6$ and $\varepsilon_7$.
The minimizing element $h\in(\mathrm{Pol}_{m-1})^\perp$ must now be 
a constant multiple of the \emph{true monic orthogonal polynomial $P$}, that is,
\[
\inf_{h\in (\mathrm{Pol}_{m-1})^\perp}
\big\|P^\ast-h\big\|_{mQ}=\big\|P^\ast-\alpha P\big\|_{mQ},\quad
\text{where}\quad
\alpha=\|P\|^{-2}_{mQ}\langle P^\ast,P\rangle_{mQ},
\]
simply as a result of the Pythagorean theorem.
Later on, in Subsections \ref{subsec-B} and \ref{subsec-AF}, 
we shall see that it is possible to arrange that
\begin{multline}
\|\Afun^\approx\|_{L^\infty(\T(0,\rho_1))}=\Ordo(1), \quad 
\|\Ffun^\approx\|_{L^\infty(\T(0,\rho_1))}=\Ordo(1), \quad 
\\
\|\Bfun^\approx\|_{L^\infty(\calA_0)}=\Ordo(1),\quad
\|\calE_m\|_{L^\infty(\calA_0)}=\Ordo(\e^{-\epsilon\sqrt{m}})
\label{eq:duality004}
\end{multline}
as $m\to+\infty$, for some constant $\epsilon>0$ which only 
depends on $\Rfun$. 
This assertion is contained in Theorem \ref{lem:collect}, and 
also expressed explicitly in \eqref{eq:Test0.05},
\eqref{eq:Bfunapproxdef2}, and \eqref{eq:AF0}, provided that the
annulus $\calA_0$ is thin enough and $\rho_1<1$ is sufficiently close to
$1$
(the control of the norm in $\Hsp_{\frac12\sigma_0}$ implies 
uniform control on an annulus about $\T$). 
Since the rate of decay
$\e^{-\epsilon\sqrt{m}}$ is slower than exponential decay, we conclude from
\eqref{eq:duality003} and \eqref{eq:duality004} that
\begin{equation}
\big\|P^\ast-\alpha P\big\|_{mQ}=\Ordo(c_m\,m^{\frac12} 
\e^{-\epsilon\sqrt{m}})
\label{eq:PstarP}
\end{equation}
where $\alpha=\|P\|_{mQ}^{-2} \langle P^\ast,P\rangle_{mQ}$. Next,
according to the Bernstein-Walsh estimate \eqref{eq:BW1}, 
it follows that
\[
|P^\ast(z)-\alpha P(z)|=\Ordo\big(c_m m^{\frac12}
\e^{-\epsilon\sqrt{m}}\e^{m\check Q_1(z)}\big).
\]
uniformly in $z$ as $m\to+\infty$. Since both $P^\ast$ and $P$ are monic,
we find that
\[
|1-\alpha|=\lim_{|z|\to+\infty}\bigg|\frac{P^\ast(z)-\alpha P(z)}{z^m}\bigg|
=\Ordo(m^{\frac12}\e^{-\epsilon\sqrt{m}}),
\]
which means that $\alpha$ is very close to $1$. Finally, to appreciate the 
value of \eqref{eq:PstarP}, we should have some general idea of the size
of the norm $\|P^\ast\|_{mQ}$.  In view of 
\eqref{eq:solF0.11}, we 
have that $\Ffun^\approx=a_1\Hfun_\Rfun+\Ordo(m^{-1})$, a fact which 
relies on the properties of $\Bfun^\approx$.
Actually, this is asserted by Theorem \ref{lem:collect} in
the Subsection \ref{subsec-AF} below (cf. equation \eqref{eq:Ffunest0}).
Using the norm identity \eqref{eq:orthog4}, this information gives that
\begin{equation}
\|\chi_{1,1}P^\approx\|_{mQ}\asymp c_m m^{-\frac14},
\label{eq:normapprox00}
\end{equation}
and if we use the estimate \eqref{eq:uest1}, it follows that
\begin{equation*}
\|P^\star\|_{mQ}\asymp c_m m^{-\frac14}. 
\end{equation*}
The formula 
\[
P=\alpha^{-1}P^\ast-\alpha^{-1}(P^\ast-\alpha P)
\]
now shows that 
\begin{equation*}
\|P-P^\ast\|_{mQ} =\Ordo\big(c_m m^{\frac12}\e^{-\epsilon\sqrt{m}}\big)
\end{equation*}
as $m\to+\infty$. We can easily turn this into a pointwise estimate by 
applying the Bernstein-Walsh inequality \eqref{eq:BW1}. However, 
it is more appealing to relate $P$ directly to the approximate polynomial
$P^\approx$ which the algorithm provides. In view of the exponential 
decay estimate \eqref{eq:uest1} together with the basic estimate 
\eqref{eq:duality004},  we have as well
\begin{equation}
\|P-\chi_{1,1}P^\approx\|_{mQ} =
\Ordo\big(c_m m^{\frac12}\e^{-\epsilon\sqrt{m}}\big)
\label{eq:normapprox02}
\end{equation}
as $m\to+\infty$. For instance, this means that in the sense of the $L^2_{mQ}$-norm,
the polynomial $P$ is small where $\chi_{1,1}=0$.
The Bernstein-Walsh estimate of Proposition 2.2.2 in 
\cite{HW1} applies here as well, although $P-\chi_{1,1}P^\approx$
is not a polynomial. It gives that
\begin{equation}
|P(z)-P^\approx(z)| =
\Ordo\big(c_m m\,\e^{-\epsilon\sqrt{m}}\e^{m\check Q_1(z)}\big)
\label{eq:normapprox03}
\end{equation}
holds uniformly as $m\to+\infty$ over the whole region where $\chi_{1,1}=1$
(this is the image under $\varphi$ of a closed exterior disk with a radius 
$\rho_2$ such that $\rho_1<\rho_2<1$), 
except that we would need to take a step away from the boundary of size 
at least $m^{-\frac12}$.  We rewrite \eqref{eq:normapprox03} in the form
\begin{multline}
P(z)=P^\approx(z)+
\Ordo\big(c_m m\,\e^{-\epsilon\sqrt{m}}\e^{m\check Q_1(z)}\big)
\\
=c_m\phi^m\e^{m\calQ}\big(F^\approx +\Ordo(m\,\e^{-\epsilon\sqrt{m}}
\e^{mR_-})\big),
\label{eq:normapprox04}
\end{multline}
which holds uniformly on the same region where $\chi_{1,1}=1$ as $m\to+\infty$, where 
$R_-=\Rfun_-\circ\phi=\check Q_1-\breve Q_1$.
If we consider the region 
\begin{equation}
D_m := \big\{z\in\D_\e(0,\rho_2):\,\,\Rfun_-(z)\le\tfrac12\epsilon m^{-\frac12}\big\},
\label{eq:regionDm}
\end{equation}
which extends inward from the exterior disk $\D_\e$ a distance proportional to
$m^{-\frac14}$, we obtain from \eqref{eq:normapprox04} that 
\begin{equation}
P(z)
=c_m\phi^m\e^{m\calQ}\big(F^\approx +\Ordo(m\,\e^{-\frac12\epsilon\sqrt{m}})
\big),\qquad z\in D_m,
\label{eq:normapprox05}
\end{equation}
uniformly as $m\to+\infty$.

\subsection{The proof of the main theorem}
We are now prepared to obtain Theorem \ref{thm:main}. Please note that the proof 
relies crucially on Theorem \ref{lem:collect}, which is established in Section \ref{sec:AGA} 
below.

\begin{proof}[Proof of Theorem \ref{thm:main}]
The asserted $L^2_{mQ}$-norm approximation of $P$ by $\chi_{1,1}P^\approx$ is 
expressed in the estimate \eqref{eq:normapprox02}, while the uniform control along
shrinking domains is asserted in \eqref{eq:normapprox05}. The  approximate size
of the norm of $\chi_{1,1}P^\approx$ is found in \eqref{eq:normapprox00}.
\end{proof}

\section{Asymptotic growth analysis}
\label{sec:AGA}

\subsection{Discrete analysis of Nishida-Nirenberg type} 
If we write $s:=m^{-1}$, which we think of as tending
to $0^+$, the equation \eqref{eq:NN1} assumes the form
\begin{equation}
\Bfun=\Bfun_{0}+s\Top[\Bfun],
\label{eq:NN2}
\end{equation}
which we may think of as a discrete one time-step
Cauchy-Kovalevskaya evolution equation, where $s$ is the
time step. The continuous Cauchy-Kovalevskaya evolution was
analyzed by Nirenberg \cite{Nir55} and Nishida \cite{Nish77}.
Nishida uses a continuous Banach scale for his solution, which
suggests we should employ a discrete collection of Banach
spaces for our one-step problem. The operators $\Top$ and
$\Lop$ do not involve the parameter $s$, and iteration of
\eqref{eq:NN2} would give that
\begin{equation}
\Bfun=\Bfun_{0}+s\Top[\Bfun_0]+\ldots+s^N\Top^N[\Bfun_0]
+s^{N+1}\Top^{N+1}[\Bfun],
\label{eq:NN3}
\end{equation}
which we think of as saying that
\begin{equation*}
\Bfun=\Bfun_{0}+s\Top[\Bfun_0]+\ldots+s^N\Top^N[\Bfun_0]
+\Ordo(s^{N+1}),
\end{equation*}
i.e., an asymptotic expansion of $\Bfun$ in $s=m^{-1}$. 
Once the asymptotic expansion is found, a suitable
function $\Bfun$  may be concocted from the
(possibly divergent) expansion, as in \cite{HedAM91}.
The functions involved, $\Hfun_\Rfun$, $\hat\Rfun$, and $W_\Rfun$
are all real-analytic in a neighborhood of $\T$, and hence the
same applies to $\Bfun_0$. To know better how to forge an approximate
solution $\Bfun^\approx$, we first analyze carefully the operators $\Top$ and 
$\Lop$.

\subsection{Polarization and the scale of Banach spaces}
We follow the approach in \cite{HW1}, and analyze real-analytic
functions $f$ in a neighborhood of $\T$ in terms of
polarized extensions $f^\diamond(z,w)$,
where $z,w$ are both near
the circle $\T$ and close to one another. The polarization
$f^{\diamond}(z,w)$ is
holomorphic in $(z,\bar w)$, and the defining property is
that $f^{\diamond}(z,z)=f(z)$. For $0<\rho<1$, we
consider annuli of the form
\[
\mathbb{A}(\rho):=\big\{z\in\C:\,\,\rho<|z|<\rho^{-1}\big\},
\]
as well as polarized
neighborhoods of $\T$ of the form
\[
\hat{\mathbb{A}}(\rho,\sigma):=\big\{(z,w)\in\C^2:\,\,
z,w\in \mathbb{A}(\rho),\,\,
|z-w|<2\sigma\big\},
\]
where $0<\sigma<1$. We shall let $\rho$ be connected to the
parameter $\sigma$ via
\begin{equation}
\rho=\rho(\sigma):=\frac{1}{\sigma+\sqrt{1+\sigma^2}}<1,
\label{eq:rhoformula0}
\end{equation}
so that asymptotically as $\sigma\to0^+$, we have that
$\rho(\sigma)=1-\sigma+\Ordo(\sigma^2)$.
For this case, we simplify the notation for the polarized
domain:
\[
\hat{\mathbb{A}}(\sigma):=\hat{\mathbb{A}}(
\rho(\sigma),\sigma).
\]
We consider a suitable scale of Banach spaces
$\Hsp^\infty_{\sigma}$ indexed by $\sigma$. To be precise,
the space $\Hsp^\infty_{\sigma}$ consists of all $L^\infty$
functions on $\hat{\mathbb{A}}(\sigma)$ that are holomorphic
in the variables $(z,\bar w)$, supplied with the
$L^\infty$ norm. Then $\Hsp^\infty_{\sigma}$ is a Banach algebra,
since $\|fg\|_{\Hsp^\infty_\sigma}\le\|f\|_{\Hsp^\infty_\sigma}
\|g\|_{\Hsp^\infty_\sigma}$ for $f,g\in \Hsp^\infty_\sigma$.
In terms of unique analytic continuation,
it is clear that the spaces $\Hsp^\infty_\sigma$
get smaller as $\sigma$ increases, that is,
$\Hsp^\infty_{\sigma}\subset \Hsp^\infty_{\sigma'}$ holds for
$0<\sigma'<\sigma<1$, and that the injection mapping
$\Hsp^\infty_{\sigma}\hookrightarrow \Hsp^\infty_{\sigma'}$ is norm
contractive. We begin with a sufficiently small $\sigma_0$ with
$0<\sigma_0<1$ such that the functions
\[
\frac{1}{\bar\Hfun_\Rfun},\quad
\bar\Hfun_\Rfun\frac{1-|\zeta|^2}{4\hat\Rfun\dbar\hat\Rfun},
\quad \frac{W_\Rfun}{\zeta\Hfun_\Rfun}
\]
appearing in the definitions of $\Lop$ and $\Top$ may be
thought of as elements of $\Hsp^\infty_{\sigma_0}$.
How small $\sigma_0$ will need to be then naturally will depend
on $\Rfun$.
In view of \eqref{eq:NN1}, the operator $\Top$ consists of
first taking the $\dbar$-derivative, then multiplying by
$\bar\Hfun_\Rfun^{-1}$, and finally applying the operator
$\Lop$. In its turn, the main ingredients for the operator
$\Lop$ are $\Pop_{H^2_-}$ and the operator $\Mop$ given by
\[
\Mop[f]:=\frac{f-\Uop_{\D_\e}[f]}{1-|\zeta|^2}.
\]
Indeed, the direct relationship reads
\begin{equation}
\Lop[f]:=\frac{(1-|\zeta|^2)\bar\Hfun_\Rfun}
{4\hat\Rfun\dbar\hat\Rfun}
\Mop[f]
-(2\pi)^{-\frac12}
\frac{W_\Rfun}{\zeta\Hfun_\Rfun}\Pop_{H^2_-}[f].
\label{eq:Lopdef1.002}
\end{equation}

\subsection{The basic growth estimate}
The basic result in this section consists of the following couple of 
estimates.

\begin{thm}
If $\sigma_0>0$ is as in the preceding subsection,
then there exists a positive constant $M_1$, depending on $\Rfun$,
such that the following estimates hold for $f\in\Hsp^\infty_{\sigma_0}$:
\begin{equation*}
\|(\Top^k[f])^\diamond\|_{\Hsp^\infty_{\frac12\sigma_0}}\le
M_1^k\,k^{2k}
\|f^\diamond\|_{\Hsp^\infty_{\sigma_0}},\qquad
\end{equation*}
and
\begin{equation*}
\|\dbar_w(\Top^{k-1}[f])^\diamond\|_{\Hsp^\infty_{\frac12\sigma_0}}\le
6\,M_1^{k-1}k^{2k-1}
\|f^\diamond\|_{\Hsp^\infty_{\sigma_0}}.
\end{equation*}
\label{lem:basicest}
\end{thm}

In the lemma, we understand the value of $k^{2k}$ to be equal to $1$ for $k=0$.

We proceed with the necessary steps to obtain the lemma.

%
\subsection{Control of the $\dbar$-derivative} In terms of
polarizations, the $\dbar$-derivative is the operator
$\dbar_w$, the holomorphic differentiation with respect
to the variable $\bar w$. For a given function
$f\in \Hsp^\infty_\sigma$, we may think of $z$ as fixed
in the smaller annulus $\mathbb{A}(\rho')$, where
$\rho'=\rho(\sigma')$ 
is given by \eqref{eq:rhoformula0} and 
$0<\sigma'<\sigma\le\sigma_0<1$.
The condition that $(z,w)\in\hat{\mathbb{A}}(\sigma')$
entails that $w$ is at distance at least
$\rho'-\rho$ to the boundary of the $w$-slice
of $\hat{\mathbb{A}}(\sigma)$ where we know that $f(z,w)$
is $\bar w$-holomorphic and bounded. Moreover, a
simple calculation gives that
\begin{equation}
\rho'-\rho\ge\frac16(\sigma-\sigma'),
\label{eq:sigmasigma'}
\end{equation}
and hence
the standard Cauchy estimates for the derivative show
that for $f\in \Hsp^\infty_\sigma$,
\begin{equation}
\|\dbar_w f^\diamond\|_{\Hsp^{\infty}_{\sigma'}}
\le 6\frac{\|f\|_{\Hsp^\infty_\sigma}}{\sigma-\sigma'},
\qquad 0<\sigma'<\sigma\le\sigma_0.
\label{eq:dbarest6}
\end{equation}

%
\subsection{Control of the operator $\Lop$}
For $f$ with polarization
$f^\diamond\in \Hsp^\infty_{\sigma}$, the restriction of
$f$ to $\T$ has a bounded
holomorphic extension to the annulus $\rho<|z|<\rho^{-1}$
with $\rho=\rho(\sigma)$ (see \eqref{eq:rhoformula0}) 
provided by the formula
\[
f_\T(z):=f^\diamond\bigg(z,\frac{1}{\bar z}\bigg).
\]
The Cauchy integral formula allows us to split uniquely
\[
f_\T(z)=f^+_\T(z)+f_\T^-(z),                                                           
\]
where $f^+_\T(z)$ is bounded and holomorphic in the disk                                          
$|z|<\rho^{-1}$ with $f_+(0)=0$ while $f^-_\T(z)$ 
is bounded and holomorphic in the exterior disk
$|z|>\rho$. Indeed, we may estimate the norms:
\begin{equation}
\|f^\pm_\T\|_{L^\infty(\mathbb{A}(\rho))}\le
\frac{2\|f_\T\|_{L^\infty(\mathbb{A}(\rho))}}{1-\rho^2}
\le
\frac{3\|f^\diamond\|_{\Hsp^\infty_\sigma}}{\sigma},\qquad
\rho=\rho(\sigma).
\label{eq:pmest1}
\end{equation}
Next, we identify
\begin{equation}
\Pop_{H^2_-}[f]=f^-_\T,
\label{eq:fminus0}
\end{equation}
and, moreover, we see that
the harmonic extension $\Uop_{\D_\e}[f]$
is given by 
\[
\Uop_{\D_\e}[f](z)=f^+_\T(1/\bar z)+f_\T^-(z),
\]
which is bounded and harmonic for $|z|>\rho_0$. Hence the
polarization of the harmonic extension is given by
\[
(\Uop_{\D_\e}[f])^\diamond(z,w)=f^+_\T(1/\bar w)+f_\T^-(z),
\]
so that the function
\begin{equation*}
F(z,w):=(\Uop_{\D_\e}[f])^\diamond(z,w)-f^\diamond(z,w)
=f^+_\T(1/\bar w)+f_\T^-(z)-f^\diamond(z,w)
\end{equation*}
is in $\Hsp^\infty_{\sigma}$ and 
vanishes on the complex variety $1-z\bar w=0$.
In view of \eqref{eq:pmest1}, its norm is easily
controlled:
\begin{equation}
\|F\|_{\Hsp^\infty_{\sigma}}\le
\frac{7\|f\|_{\Hsp^\infty_{\sigma}}}{\sigma}.
\label{eq:Fest1}
\end{equation}
It follows from the Weierstrass division theorem that
the function
\[
G(z,w):=\frac{F(z,w)}{1-z\bar w}
=\frac{(\Uop_{\D_\e}[f])^\diamond(z,w)-f^\diamond(z,w)}
{1-z\bar w}
\]
is holomorphic in the variables $(z,\bar w)$ on
$\hat{\mathbb{A}}(\sigma)$. If $w\in\mathbb{A}(\rho')$
where $\rho'=\rho(\sigma')$ and $0<\sigma'<\sigma$,
we see that
\[
|1-z\bar w|\ge \rho'-\rho,\qquad |z|\in\{\rho,\rho\}.
\]
In view of \eqref{eq:Fest1}, \eqref{eq:sigmasigma'},
and the  maximum principle, then, we obtain that
\[
\|(\Mop[f])^\diamond\|_{\Hsp^\infty_{\sigma'}}=
\|G\|_{\Hsp^\infty_{\sigma'}}\le
\frac{42\,\|f^\diamond\|_{\Hsp^\infty_{\sigma}}}
{\sigma(\sigma-\sigma')}.
\]
This is the main expression to control in the definition
of the operator $\Lop$, the other being controlled by
\eqref{eq:pmest1} in view of the identity
\eqref{eq:fminus0}. Together, these estimates show that
\begin{equation}
\|(\Lop[f])^\diamond\|_{\Hsp^\infty_{\sigma'}}\le
\frac{L_1\|f^\diamond\|_{\Hsp^\infty_{\sigma}}}
{\sigma(\sigma-\sigma')},\qquad
0<\sigma'<\sigma\le\sigma_0<1,
\label{eq:Lest.001}
\end{equation}
for some positive constant $L_1$.

%
\subsection {Estimation of the operator $\Top$}
We think of $\dbar$ as acting
$\Hsp^\infty_\sigma\to \Hsp^\infty_{\sigma''}$ whereas $\Lop$
acts $\Hsp^\infty_{\sigma''}\to \Hsp^\infty_{\sigma'}$, where
$0<\sigma'<\sigma''<\sigma\le\sigma_0$.
A combination of the estimates \eqref{eq:dbarest6}
and \eqref{eq:Lest.001} shows that
\begin{equation*}
\|(\Top[f])^\diamond\|_{\Hsp^\infty_{\sigma'}}\le
\frac{L_2\|f^\diamond\|_{\Hsp^\infty_{\sigma}}}
{\sigma''(\sigma-\sigma'')(\sigma''-\sigma')},\qquad
0<\sigma'<\sigma''<\sigma\le\sigma_0<1,
\end{equation*}
for some positive constant $L_2$. We may insert
the choice
$\sigma''=\frac12(\sigma+\sigma')$ and use the fact
that $\sigma''\ge\sigma'$, to obtain the estimate
\begin{equation}
\|(\Top[f])^\diamond\|_{\Hsp^\infty_{\sigma'}}\le
\frac{4L_2\|f^\diamond\|_{\Hsp^\infty_{\sigma}}}
{\sigma'(\sigma-\sigma')^2},\qquad
0<\sigma'<\sigma\le\sigma_0<1.
\label{eq:Test0.01}
\end{equation}


\subsection{The proof of the basic growth estimate}
We may now obtain Theorem \ref{lem:basicest}.

\begin{proof}[Proof of Theorem \ref{lem:basicest}]
We consider a strictly decreasing finite sequence of positive
reals $\sigma_j$, $j=0,1,\ldots,k$, where $\sigma_0$
is as before.
and we require that 
$\sigma_k=\frac12\sigma_0$.
We apply \eqref{eq:Test0.01} with $\sigma=\sigma_j$ and
$\sigma'=\sigma_{j+1}$ and iterate. Since
$\sigma_k=\frac12\sigma_0$, we obtain successively
\begin{equation}
\big\|(\Top^k[f])^\diamond\big\|_{\Hsp^\infty_{\frac{1}{2}\sigma_0}}\le
\bigg(\frac{8L_2}{\sigma_0}\bigg)^k
\frac{\|f^\diamond\|_{\Hsp^\infty_{\sigma_0}}}
{\prod_{j=0}^{k-1}(\sigma_{j}-\sigma_{j+1})^2}.
\label{eq:Test0.02}
\end{equation}
By combining further with the Cauchy estimate 
\eqref{eq:dbarest6}, 
we find that
\begin{equation}
\big\|\dbar_w(\Top^{k-1}[f])^\diamond\big\|_{\Hsp^\infty_{\frac12\sigma_{0}}}\le
6\bigg(\frac{8L_2}{\sigma_0}\bigg)^{k-1}
\frac{\|f^\diamond\|_{\Hsp^\infty_{\sigma_0}}}
{(\sigma_{k-1}-\sigma_{k})\prod_{j=0}^{k-2}(\sigma_{j}-\sigma_{j+1})^2}.
\label{eq:Test0.02'}
\end{equation}
We make the simplest choice possible, all the $\sigma_j$ in
an arithmetic progression:
\[
\sigma_j=\sigma_0-\frac{j\sigma_0}{2k},\qquad j=0,\ldots,k.
\]
With this choice, we read off from \eqref{eq:Test0.02} that
\begin{equation*}
\|(\Top^k[f])^\diamond\|_{\Hsp^\infty_{\frac12\sigma_0}}\le
M_1^k\,k^{2k}
\|f^\diamond\|_{\Hsp^\infty_{\sigma_0}},\qquad
M_1:=\frac{32\,L_2}{\sigma_0^3},
\end{equation*}
as claimed.
This tells us how to pick $M_1$ in terms of $L_2$ and $\sigma_0$, and
each depends only on the properties of $\Rfun$. 
Finally, as for \eqref{eq:Test0.02'}, we
obtain
\begin{equation*}
\|\dbar_w(\Top^{k-1}[f])^\diamond\|_{\Hsp^\infty_{\frac12\sigma_0}}\le
6\,M_1^{k-1}k^{2k-1}
\|f^\diamond\|_{\Hsp^\infty_{\sigma_0}},
\end{equation*}
which completes the proof. 
\end{proof}

%
\subsection{The choice of an approximate function $\Bfun$}
\label{subsec-B}
The first estimate in Theorem \ref{lem:basicest} says that in the Taylor
expansion \eqref{eq:NN3}, 
the norm growth of the coefficients in
$\Hsp^\infty_{\frac12\sigma_0}$
is such that as a function of the parameter $s=m^{-1}$,
the function $\Bfun$ behaves as a  if it were from the Gevrey class with
exponent $3$. We now make an effort to find a universal
representative $\Bfun$. We recall from \eqref{eq:NN3} that
the asymptotic expansion for $\Bfun$ should be
\[
\Bfun\sim \Bfun_0+s\Top[\Bfun_0]+s^2\Top^2[\Bfun_0]+\ldots,
\]
where $s=m^{-1}$. The approach in \cite{HedAM91} is based on
the idea to work with the convergent series
\begin{equation*}
\Bfun\sim\sum_{k=0}^{+\infty}s^k \mu_k(s)\Top^k[\Bfun_0],
\end{equation*}
where the multiplier functions $\mu_k(s)$ are chosen
appropriately to ensure convergence for any $s$, while
for an individual index $k$, $\mu_k(s)=1+\Ordo(s^N)$ holds for
any finite positive integer $N$ as $s\to0$.  
This is natural when we would like to have some degree of 
smoothness in the parameter $s=m^{-1}$, for instance when
considering Gevrey classes, but in our context,
$m$ is a positive integer, and $s$ therefore discrete.
A simpler approach which we employ here is is to consider finite cut-offs
(``abschnitts'') $\Bfun^{\langle\kappa-1\rangle}$ given by
\begin{equation}
\Bfun^\approx:=\Bfun^{\langle\kappa-1\rangle}=\Bfun_0+
s\Top[\Bfun_0]+\ldots+s^{\kappa-1}\Top^{\kappa-1}[\Bfun_0],
\label{eq:Bfunapproxdef}
\end{equation}
where the positive integer parameter $\kappa$ will be taken to depend on 
$s=m^{-1}$. We then calculate the error term $\calE_m$ from the
equality \eqref{eq:approxB} in reverse:
\begin{equation*}
\calE_m=\Bfun^\approx-\Bfun_0-s\Top[\Bfun^\approx]=-s^{\kappa}
\Top^{\kappa}[\Bfun_0]=-m^{-\kappa}\Top^\kappa[\Bfun_0].
\end{equation*}
Now, in view of the first estimate of Theorem \ref{lem:basicest}, we have
\begin{equation}
\|\calE_m^\diamond\|_{\Hsp^\infty_{\frac12\sigma_0}}
=m^{-\kappa}
\|(\Top^\kappa[\Bfun_0])^\diamond\|_{\Hsp^\infty_{\frac12\sigma_0}}
\le M_1^\kappa m^{-\kappa}\,\kappa^{2\kappa}
\|\Bfun_0^\diamond\|_{\Hsp^\infty_{\sigma_0}}.
\label{eq:Test0.04}
\end{equation}
It is now a calculus exercise to minimize over the parameter $\kappa$.
To this end, we provide a lemma.

\begin{lem}
For a real parameter $\beta$, consider the function $\eta(t):=\beta t+2t\log t$ 
for $0<t<+\infty$, while $\eta(0):=0$, and let $t_0$ denote the positive real number
$t_0:=\e^{-\frac12\beta-1}$. Then $\eta(t)$ enjoys the estimate 
 \[
 \eta(t)\le
 \begin{cases}
 -2t,\qquad\qquad 0\le t\le t_0,
 \\
 -2t_0+t_0^{-1},\quad t_0\le t\le t_0+1.
 \end{cases}
 \]
\label{lem:calcex}
\end{lem}

\begin{proof}
We calculate the first two derivatives: 
\[
\eta'(t)=\beta+2+2\log t\quad\text{and}\quad \eta''(t)=2/t.
\]
In particular, it follows that $\eta(t)$ is convex. 
We look for critical points, and find the unique critical point 
$t_0=\e^{-\frac12\beta-1}$ with value
$\eta(t_0)=-2t_0=-2\e^{-\frac12\beta-1}$. At the point $t=t_0$, the function 
$\eta(t)$ has a global minimum.
We attempt to estimate the growth of $\eta(t)$ for $t>t_0$. Since $\eta''(t)$
is decreasing, the integration-by-parts formula shows that
\begin{multline*}
\eta(t)=\eta(t_0)+\int_{t_0}^t (x-t)\eta''(x)\diff x\le 
\eta(t_0)+\int_{t_0}^t (x-t)\eta''(t_0)\diff x
\\
=\eta(t_0)+\frac{\eta''(t_0)}{2}(t-t_0)^2,\qquad t\ge t_0.
\end{multline*}
We implement this for $t$ in the interval $[t_0,t_0+1]$ and obtain
\begin{equation*}
\eta(t)\le \eta(t_0)+\frac{\eta''(t_0)}{2}=-2t_0+t_0^{-1},
\qquad t_0\le t\le t_0+1.
\end{equation*}
Also, by convexity and the fact that $\lim \eta(t)=0$ as $t\to0^+$, 
we have the estimate
\begin{equation*}
\eta(t)\le \frac{\eta(t_0)}{t_0}t=-2t,\qquad 0\le t\le t_0.
\end{equation*}
which completes the proof of the lemma.
\end{proof}

We shall apply the lemma to the function
\[
\eta(k)=\log(M_1^k m^{-k}k^{2k})= \beta k+2k\log k\quad
\text{with}\quad \beta=\log M_1-\log m.
\]
Then $t_0=\e^{-\frac12\beta-1}=\e^{-1}M_1^{-\frac12}m^{\frac12}$
and it follows from Lemma \ref{lem:calcex} that if $\kappa_\ast(m)$ 
denotes the unique integer in the interval $[t_0,t_0+1[$,
\[
\eta(\kappa_\ast(m))\le -2t_0+t_0^{-1}=-2\e^{-1}M_1^{-\frac12}m^{\frac12}
+\e\,M_1^{\frac12}m^{-\frac12}\le -\epsilon m^{\frac12}+\log2,
\]
where $\epsilon:=2\e^{-1}M_1^{-\frac12}$, provided that $m$ is big enough, 
in this case $m\ge 25 M_1$ suffices. 
After exponentiation, we find that 
\begin{equation}
M_1^{\kappa_\ast(m)} m^{-\kappa_\ast(m)}\,(\kappa_\ast(m))^{2\kappa_\ast(m)}
\le 2\exp(-\epsilon\sqrt{m}).
\label{eq:esteta3}
\end{equation}
\emph{This parameter value $\kappa=\kappa_\ast(m)$ is 
the choice we insert into the definition} \eqref{eq:Bfunapproxdef}. It now
follows from \eqref{eq:Test0.04} and \eqref{eq:esteta3} that
with  $\epsilon=2\e^{-1}M_1^{-\frac12}$,
\begin{multline}
\|\calE_m^\diamond\|_{\Hsp^\infty_{\frac12\sigma_0}}
=m^{-\kappa_\ast(m)}
\|(\Top^\kappa[\Bfun_0])^\diamond\|_{\Hsp^\infty_{\frac12\sigma_0}}
\\
\le M_1^{\kappa_\ast(m)}m^{-\kappa_*(m)}(\kappa_\ast(m))^{2\kappa_\ast(m)}
\le2\e^{-\epsilon\sqrt{m}}\|\Bfun_0^\diamond\|_{\Hsp^\infty_{\sigma_0}},
\label{eq:Test0.05}
\end{multline}
provided that $m\ge 25 M_1$. This controls the error term $\calE_m$, but
we will also need to estimate $\Bfun^\approx$ itself. By the triangle inequality
and the first estimate of Theorem \ref{lem:basicest}, combined with the first
estimate of Lemma \ref{lem:calcex}
with $\beta=\log M_1-\log m$,
\begin{multline}
\|(\Bfun^\approx)^\diamond\|_{\Hsp^\infty_{\frac12\sigma_0}}
\le\sum_{j=0}^{\kappa_\ast(m)-1}
m^{-j}\big\|(\Top^{j}[\Bfun_0])^\diamond\big\|_{\Hsp^\infty_{\frac12\sigma_0}}
\\
\le\sum_{j=0}^{\kappa_\ast(m)-1}
M_1^j m^{-j}j^{2j} \|\Bfun_0^\diamond\|_{\Hsp^\infty_{\sigma_0}}
=\sum_{j=0}^{\kappa_\ast(m)-1}
\e^{\eta(j)} \|\Bfun_0^\diamond\|_{\Hsp^\infty_{\sigma_0}}
\\
\le\sum_{j=0}^{\kappa_\ast(m)-1}
\e^{-2j} \|\Bfun_0^\diamond\|_{\Hsp^\infty_{\sigma_0}}\le
\frac{\|\Bfun_0^\diamond\|_{\Hsp^\infty_{\sigma_0}}}{1-\e^{-2}},
\label{eq:Bfunapproxdef2}
\end{multline}
so the norm of $\Bfun^\approx$ is uniformly bounded. To control 
$\dbar\Bfun^\approx$ as well, we may rely instead on the second estimate 
of Theorem \ref{lem:basicest}.
We find that
\begin{multline}
\big\|\dbar_w(\Bfun^\approx)^\diamond\big\|_{\Hsp^\infty_{\frac12\sigma_0}}
\le\sum_{j=0}^{\kappa_\ast(m)-1}
m^{-j}\big\|\dbar_w(\Top^{j}[\Bfun_0])^\diamond\big\|_{\Hsp^\infty_{\frac12\sigma_0}}
\\
\le\sum_{j=0}^{\kappa_\ast(m)-1}
M_1^j m^{-j}(j+1)^{2j+1} \|\Bfun_0^\diamond\|_{\Hsp^\infty_{\sigma_0}}
\le \e^2\sum_{j=0}^{\kappa_\ast(m)-1}
(j+1)M_1^j m^{-j}j^{2j} \|\Bfun_0^\diamond\|_{\Hsp^\infty_{\sigma_0}}
\\
=\e^2\sum_{j=0}^{\kappa_\ast(m)-1}
(j+1)\,\e^{\eta(j)} \|\Bfun_0^\diamond\|_{\Hsp^\infty_{\sigma_0}}
\le\e^2\sum_{j=0}^{\kappa_\ast(m)-1}
(j+1)\e^{-2j} \|\Bfun_0^\diamond\|_{\Hsp^\infty_{\sigma_0}}\le
\frac{\e^2\|\Bfun_0^\diamond\|_{\Hsp^\infty_{\sigma_0}}}{(1-\e^{-2})^2},
\label{eq:Bfunapproxdef3}
\end{multline}
where in the middle we used the elementary estimate $(j+1)^j\le\e\,j^j$
which we recognize as $(1+j^{-1})^j\le \e$. This estimate holds trivially 
for $j=0$ since we interpret $0^0$ as $1$.



\smallskip

\subsection{The approximate functions $\Afun^\approx$ and $\Ffun^\approx$}
\label{subsec-AF}
The equations \eqref{eq:Abeta1} and \eqref{eq:solF0.11} give us $\Afun^\approx$
and $\Ffun^\approx$ in terms of $\Bfun^\approx$, which in turn is given by
\eqref{eq:Bfunapproxdef} with $\kappa=\kappa_\ast(m)$, where $\kappa_\ast(m)$ 
is as in the preceding subsection. The function $\dbar\Bfun^\approx$ gets estimated
by \eqref{eq:Bfunapproxdef3}, and division by $\bar \Hfun_\Rfun$ is fine since
its polarization is in $\Hsp^\infty_{\sigma_0}$ by assumption and it is bounded away
from $0$ in a fixed polarized neighborhood of $\T$ (so taking $\sigma_0>0$ small
enough we may safely divide). In view of \eqref{eq:pmest1}, 
\begin{equation}
\bigg\|\Pop_{H^2_{-}}\bigg[\frac{\dbar\Bfun^\approx}{\bar\Hfun_\Rfun}\bigg]
\bigg\|_{L^\infty(\T(0,\rho_\ast))}\le\frac{6\sigma_0^{-1}\e^2}{(1-\e^{-1})^2}
\|\Bfun_0^\diamond\|_{\Hsp^\infty_{\sigma_0}}
\big\|(\bar\Hfun_\Rfun^\diamond)^{-1}\big\|_{\Hsp^\infty_{\frac12\sigma_0}}
\label{eq:Popest0.1}
\end{equation}
if the radius $\rho_\ast<1$ is given by
\[
\rho_\ast:=\rho(\tfrac12\sigma_0)=\frac{2}{\sigma_0+\sqrt{4+\sigma_0^2}}.
\]
In the same manner, we obtain from \eqref{eq:pmest1} that
\begin{equation}
\bigg\|\Pop_{H^2_0}\bigg[\frac{\dbar\Bfun^\approx}{\bar\Hfun_\Rfun}\bigg]
\bigg\|_{L^\infty(\T(0,\rho^{-1}_\ast))}\le\frac{6\sigma_0^{-1}\e^2}{(1-\e^{-1})^2}
\|\Bfun_0^\diamond\|_{\Hsp^\infty_{\sigma_0}}
\big\|(\bar\Hfun_\Rfun^\diamond)^{-1}\big\|_{\Hsp^\infty_{\frac12\sigma_0}},
\label{eq:Popest0.2}
\end{equation}
where we recall from Subsection \ref{ss:Herglotz}
that $\Pop_{H^2_0}$ denotes the orthogonal projection 
$L^2(\T)\to H^2_0$.
The identity
\[
\mathrm{conj}\,\Pop_{H^2_0}\bigg[\frac{\dbar\Bfun^\approx}{\bar\Hfun_\Rfun}\bigg]
\bigg(\frac{1}{\bar z}\bigg)
=\Pop_{H^2_{-,0}}\bigg[\frac{\partial\bar\Bfun^\approx}{\Hfun_\Rfun}\bigg](z)
\]
where "conj" denote complex conjugation, shows that \eqref{eq:Popest0.2} 
gives the equivalent estimate
\begin{equation}
\bigg\|\Pop_{H^2_{-,0}}\bigg[\frac{\partial\bar\Bfun^\approx}{\Hfun_\Rfun}\bigg]
\bigg\|_{L^\infty(\T(0,\rho_\ast))}\le\frac{6\sigma_0^{-1}\e^2}{(1-\e^{-1})^2}
\|\Bfun_0^\diamond\|_{\Hsp^\infty_{\sigma_0}}
\big\|(\bar\Hfun_\Rfun^\diamond)^{-1}\big\|_{\Hsp^\infty_{\frac12\sigma_0}}.
\label{eq:Popest0.3}
\end{equation}
We obtain as a consequence of \eqref{eq:Popest0.2} that
\begin{multline}
\bigg\|\Afun^\approx-\frac{\apar_1}{\zeta\Hfun_\Rfun}\bigg\|_{L^\infty(\T(0,\rho_\ast))}
\le m^{-1}\frac{6\sigma_0^{-1}(2\pi)^{-\frac12}\e^2}{(1-\e^{-1})^2}
\\
\times\|(\zeta\Hfun_\Rfun)^{-1}\|_{L^\infty(\T(0,\rho_\ast))}
\|\Bfun_0^\diamond\|_{\Hsp^\infty_{\sigma_0}}
\big\|(\bar\Hfun_\Rfun^\diamond)^{-1}\big\|_{\Hsp^\infty_{\frac12\sigma_0}}
=\Ordo(m^{-1}),
\label{eq:Afunest0}
\end{multline}
and, analogously, from \eqref{eq:Popest0.3} that
\begin{multline}
\big\|\Ffun^\approx-{\apar_1}{\Hfun_\Rfun}\big\|_{L^\infty(\T(0,\rho_\ast))}
\le m^{-1}\frac{6\sigma_0^{-1}(2\pi)^{-\frac12}\e^2}{(1-\e^{-1})^2}
\\
\times\|\Hfun_\Rfun\|_{L^\infty(\T(0,\rho_\ast))}
\|\Bfun_0^\diamond\|_{\Hsp^\infty_{\sigma_0}}
\big\|(\bar\Hfun_\Rfun^\diamond)^{-1}\big\|_{\Hsp^\infty_{\frac12\sigma_0}}
=\Ordo(m^{-1})
\label{eq:Ffunest0}
\end{multline}
as $m\to+\infty$. In particular,
\begin{equation}
\big\|\Afun^\approx\big\|_{L^\infty(\T(0,\rho_\ast))}=\Ordo(1)
\quad\text{and}\quad
\big\|\Ffun^\approx\big\|_{L^\infty(\T(0,\rho_\ast))}=\Ordo(1)
\label{eq:AF0}
\end{equation}
as $m\to+\infty$. 
We recall the definition of the function $\Bfun^\approx$:
\[
\Bfun^\approx=\Bfun^{\langle \kappa-1\rangle}=\Bfun_0+\ldots+
m^{-\kappa+1}\Top^{\kappa-1}\Bfun_0,\qquad \kappa:=\kappa_\ast(m),
\]
where $\kappa_\ast(m)$ is the unique integer in the interval $[t_0,t_0+1[$
with $t_0=\e^ {-1}M_1^{-\frac12}m^{\frac12}$. Also, the functions 
$\Afun^\approx$ and $\Ffun^\approx$ are given in terms of $\Bfun^\approx$ by
the relations \eqref{eq:Abeta1}  and \eqref{eq:solF0.11}, respectively. 
We gather our observations in a theorem.

\begin{thm}
With the above choices of $\Bfun^\approx$, $\Afun^\approx$, and $\Ffun^\approx$,
the size of the error term $\calE_m$ in the equality \eqref{eq:approxB} is 
controlled by the estimate \eqref{eq:Test0.05}, whereas the size of 
$\Bfun^\approx$ itself is controlled by \eqref{eq:Bfunapproxdef2} and 
\eqref{eq:Bfunapproxdef3}. As for the functions $\Afun^\approx$ and
$\Ffun^\approx$, their norms are uniformly bounded according to 
\eqref{eq:AF0}, and, to higher precision, controlled by \eqref{eq:Afunest0} and
\eqref{eq:Ffunest0}, respectively.
\label{lem:collect}
\end{thm}

This result was implemented back in Subsection \ref{ss:Horm} to establish
the necessary facts which support Theorem \ref{thm:main}.

\end{document}